\documentclass[12pt,twoside]{amsart}
\usepackage[margin=3cm]{geometry}
\usepackage[colorlinks=false]{hyperref}
\usepackage[english]{babel}
\usepackage{graphicx,titling}
\usepackage{float}
\usepackage{amsmath,amsfonts,amssymb,amsthm}
\usepackage{lipsum}
\usepackage[T1]{fontenc}
\usepackage{fourier}
\usepackage{color}
\usepackage[latin1]{inputenc}
\usepackage{esint}
\usepackage{caption}
\usepackage{ccicons}

\makeatletter
\def\blfootnote{\xdef\@thefnmark{}\@footnotetext}
\makeatother

\newcommand\ccnote{
    \blfootnote{\copyright\,\, Otis Chodosh, Nick Edelen, and Chao Li}
    \blfootnote{\ccLogo\, \ccAttribution\,\, Licensed under a \href{https://creativecommons.org/licenses/by/4.0/}{Creative Commons Attribution License (CC-BY)}.}
}

\usepackage[export]{adjustbox}
\numberwithin{equation}{section}
\usepackage{setspace}
\renewcommand{\le}{\leqslant}
\renewcommand{\leq}{\leqslant}
\renewcommand{\ge}{\geqslant}
\renewcommand{\geq}{\geqslant}
\renewcommand{\mathbb}{\varmathbb}
\usepackage{fancyhdr}
\newtheorem{theo}{Theorem}[section]
\newtheorem{lemm}[theo]{Lemma}

\newtheorem{prop}[theo]{Proposition}
\newtheorem{rema}[theo]{Remark}
\usepackage{mathrsfs}
\usepackage{comment}

\newcommand{\RR}{\mathbf{R}}
\newcommand{\R}{\mathbf{R}}

\newcommand{\cA}{\mathcal A}

\newcommand{\cH}{\mathcal H}

\newcommand{\id}{\text{id}}

\DeclareMathOperator{\graph}{graph}

\DeclareMathOperator{\spt}{spt}

\DeclareMathOperator{\Lip}{Lip}

\DeclareMathOperator{\leb}{\mathcal{L}}

\newcommand{\del}{\partial}
\DeclareMathOperator{\Vol}{Vol}

\DeclareMathOperator{\Ric}{Ric}

\DeclareMathOperator{\Div}{div}

\DeclareMathOperator{\sing}{sing}

\newcommand{\pa}[2]{\frac{\partial #1}{\partial #2}}

\newcommand{\eps}{\varepsilon}

\address{Otis Chodosh, Department of Mathematics, Stanford University, Building 380, Stanford, CA 94305, USA}
\email{ochodosh@stanford.edu}
\medskip
\address{Nick Edelen, Department of Mathematics, University of Notre Dame, Notre Dame, IN 46556 USA} 
\email{nedelen@nd.edu}
\medskip
\address{Chao Li, Courant Institute, New York University, 251 Mercer St, New York, NY 10012, USA}
\email{chaoli@nyu.edu}

\begin{document}

\thispagestyle{empty}

\begin{minipage}{0.28\textwidth}
\begin{figure}[H]
\includegraphics[width=2.5cm,height=2.5cm,left]{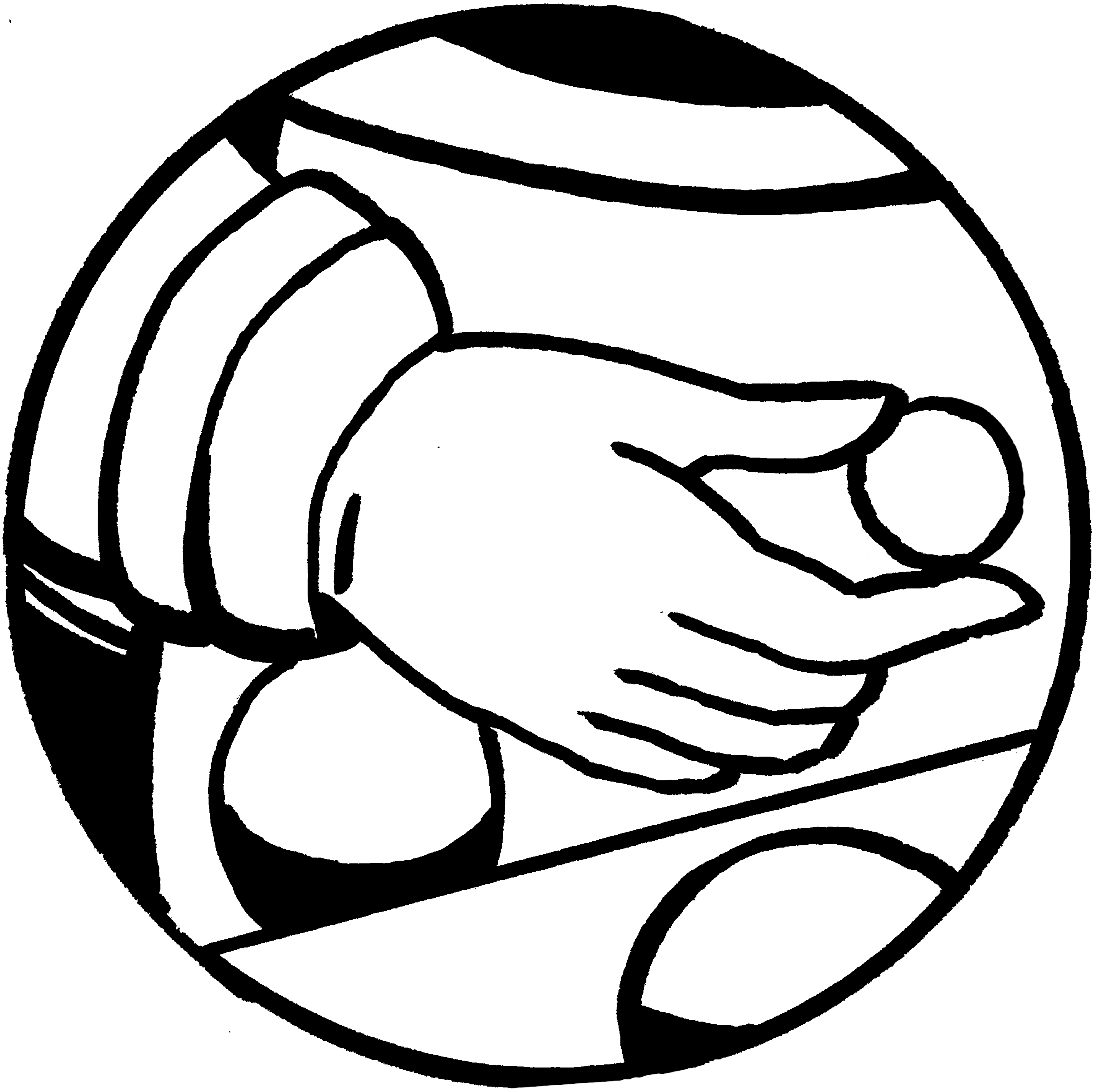}
\end{figure}
\end{minipage}
\begin{minipage}{0.7\textwidth} 
\begin{flushright}
Ars Inveniendi Analytica (2025), Paper No. 2, 27 pp.
\\
DOI 10.15781/axnt-ww91
\\
ISSN: 2769-8505
\end{flushright}
\end{minipage}

\ccnote

\vspace{1cm}


\begin{center}
\begin{huge}
\textit{Improved regularity for minimizing capillary hypersurfaces}


\end{huge}
\end{center}

\vspace{1cm}


\begin{minipage}[t]{.28\textwidth}
\begin{center}
{\large{\bf{Otis Chodosh}}} \\
\vskip0.15cm
\footnotesize{Stanford University}
\end{center}
\end{minipage}
\hfill
\noindent
\begin{minipage}[t]{.28\textwidth}
\begin{center}
{\large{\bf{Nick Edelen}}} \\
\vskip0.15cm
\footnotesize{University of Notre Dame}
\end{center}
\end{minipage}
\hfill
\noindent
\begin{minipage}[t]{.28\textwidth}
\begin{center}
{\large{\bf{Chao Li}}} \\
\vskip0.15cm
\footnotesize{Courant Institute \\ New York University} 
\end{center}
\end{minipage}

\vspace{1cm}


\begin{center}
\noindent \em{Communicated by Guido De Philippis}
\end{center}
\vspace{1cm}


\noindent \textbf{Abstract.} \textit{We  give improved estimates for the size of the singular set of  minimizing capillary hypersurfaces: the singular set is always of codimension at least $4$, and this estimate improves if the capillary angle is close to $0$, $\frac{\pi}{2}$, or $\pi$.  For capillary angles that are close to $0$ or $\pi$, our analysis is based on a rigorous connection between the capillary problem and the one-phase Bernoulli problem.}
\vskip0.3cm

\vspace{0.5cm}


\section{Introduction}

Given a smooth Riemannian $(n+1)$-manifold with boundary $(X^{n+1},h)$, functions $\sigma\in C^1(\partial X)$ satisfying $\sigma\in (-1,1)$ and $g\in C^1(X)$, and an open subset $E\subset X$, the Gauss' free energy is
\begin{equation}\label{eq:capillary.functional}
\cA (E) = \cH^n( \partial^* E \cap \mathring X) - \int_{\partial^* E \cap \partial X}\sigma(x)d\cH^{n} + \int_E g(x)dx.
\end{equation}
Stationary points (subject to a volume contraint) to the Gauss' free energy model the equilibrium state of incompressible fluids. Indeed, letting $X$ be the container and $E$ the portion of the fluid, the three terms in \eqref{eq:capillary.functional} represent the free surface energy (proportional to area of the separating surface), the wetting energy (modeling the adhesion of the fluid with the wall of the container), and the gravitational energy, respectively. The interface, $M=\partial E\cap \mathring{X}$, is usually called a \textit{capillary surface}. Formally, if $E$ is a stationary point of $\cA$, then the capillary surface $M$ should have prescribed mean curvature $g$ and boundary contact angle $\cos^{-1}(\sigma)$. However, it is well-known that minimizers of $\cA$ may have a singular set. The main result of this paper gives improved bounds for the size of this singular set.

\begin{theo}\label{thm:main.regularity}
	There exist $\varepsilon_0, \varepsilon_1>0$ such that the following holds. Let $(X^{n+1},h)$ be a smooth manifold with boundary, $\sigma\in C^1(\partial X)$, $-1<\sigma<1$, and $g\in C^1(X)$. Suppose that $E\subset X$ is a Caccioppoli set that minimizes \eqref{eq:capillary.functional}. Then $M=\spt(\partial E\cap \mathring{X})$ is a smooth hypersurface of $M$ away from a closed set $\sing(M)$. 
	\begin{enumerate}
		\item We always have that $\dim (\sing(M)\cap \partial X)\le n-4$;
		\item Let $S_1 = \{x\in \partial X: \sigma(x)\in (-1,-1+\eps_0)\cup (1-\eps_0,1)\}$. Then
		\[\dim (\sing(M)\cap S_1)\le n-5.\]
		\item Let $S_2 = \{x\in \partial X: \sigma(x)\in (-\eps_1,\eps_1)\}$. Then
		\[\dim (\sing(M)\cap S_2)\le n-7.\]
	\end{enumerate}
\end{theo}

Previously, the best known partial regularity, $\dim\sing(M)\le n-3$, was proven by Taylor in \cite{Taylor1}. This was extended to the anisotropic case (extending the free surface energy to an elliptic integral depending on the surface unit normal) by De Philippis--Maggi in \cite{DePhilippisMaggi}. When $\theta= \frac{\pi}{2}$, the surface $M$ is usually called a free-boundary minimal surface, and satisfies the sharp estimate $\dim\sing(M)\le n-7$ by Gr\"uter--Jost \cite{GruterJost}. Capillary surfaces have shown potential applications in comparison geometry (see, e.g.\ \cite{LiPolyhedron,LiDihedral}), where a major stumbling block was the lack a satisfactory boundary regularity in the form of Theorem \ref{thm:main.regularity}.

By the $\varepsilon$-regularity and compactness theorems due to De Philippis--Maggi \cite{DePhilippisMaggi}, the monotonicity formula of Kagaya-Tonegawa \cite{KagayaTonegawa} (see also \cite{DeEdGaLi}), and Federer's dimension reduction, Theorem \ref{thm:main.regularity} follows from the classification of capillary minimizing cones with an isolated singularity. Thus, for every open set $U\subset \RR^{n+1}$, let us consider the functional (recall that $\cos\theta = \sigma$ by the first variation)
\[\cA_U^\theta (\Omega) = |\partial^* \Omega \cap \RR^{n+1}_+ \cap U| - \cos\theta |\partial^* \Omega \cap \partial \RR^{n+1}_+\cap U|.\]
We denote $\cA_{\RR^{n+1}}^\theta$ simply by $\cA^\theta$, where $\RR^{n+1}_+=\{x_1>0\}$ is an open half space. We say that a Cacciopolli set $\Omega\subset \RR^{n+1}_+$ minimizes $\cA^\theta$, if for every pre-compact open set $U\subset \RR^{n+1}$,
\[ \cA_U^\theta (\Omega)\le \cA_U^\theta (\Omega'),\]
for every Cacciopolli set $\Omega'\subset \RR^{n+1}_+$ such that $\Omega\Delta\Omega'$ is compactly contained in $U$.

Theorem \ref{thm:main.regularity} is thus a consequence of the following classification result for cones. 

\begin{theo}\label{thm:main.cone}
	There exist $\theta_0, \theta_1>0$ such that the following holds. Let $\Omega\subset \RR^{n+1}_+$ be a minimizing cone for $\cA^\theta$ and  $M=\partial \Omega\cap \RR^{n+1}_+$. Assume that $M$ is smooth away from $0$. Suppose one of the following holds:
	\begin{enumerate}
		\item $n\le 3$;
		\item $n=4$, $\theta\in (0,\theta_0)\cup (\pi-\theta_0,\pi)$;
		\item $4\le n\le 6$, $\theta\in (\frac{\pi}{2}-\theta_1,\frac{\pi}{2}+\theta_1)$.
	\end{enumerate}
	Then $M$ is flat.
\end{theo}

Since any $\cA^\theta$-minimizer $\Omega$ gives a $\cA^{\pi-\theta}$-minimizer $\RR^{n+1}_+\setminus \Omega$, we may, without loss of generality, assume that $\theta\in (0,\frac{\pi}{2}]$. With this in mind, each of the three cases in Theorem \ref{thm:main.cone} is treated differently. Case  (1) follows from an Almgren-type argument as in \cite{Almgren} by analyzing the stability inequality on the spherical link $\Sigma=M\cap S^n(1)$, and a topological classification result for embedded capillary disks analogous to Fraser--Schoen \cite{FraserSchoen} in the  free-boundary case. Case (3) is analogous to the previous work by Li--Zhou--Zhu on the stable Bernstein theorem for capillary surfaces \cite[Appendix C]{LiZhouZhu}, which follows from a careful Simons-type computation as in \cite{Simons}, extending the one for stable free-boundary minimal hypersurfaces.

\subsection{Relationship with one-phase free boundary problem} The most interesting part of Theorem \ref{thm:main.cone} is Case (2), and it relies on a novel application of the connection between the capillary problem and the one-phase Bernoulli free boundary problem. To motivate the discussion, let us assume, for simplicity, that $\RR^{n+1}_+ = \{x_1>0\}$, and $M$ is contained in the graph over the $x_1$-direction of a Lipschitz function $u$ - i.e. there exists $u\in \Lip (\RR^n=\{x_1=0\})$ such that
\[M = \{(u(x'),x'): x'\in \RR^{n}, u(x')>0\},\]
here $x'=(x_2,\cdots,x_n)$. Set $v=\frac{u}{\tan\theta}$. Assuming that $|\Lip(v)|$ is uniformly bounded as $\theta\to 0$, we compute that
\begin{equation}\label{eq.capillary.vs.AC.formally}
\begin{split}
\cA^\theta (\Omega) &= \int_{\{ u > 0 \} \subset \RR^n}\left(\sqrt{1+|Du|^2} - \cos\theta 1_{\{u >0\}}\right)dx'\\
&= \int_{\{ u > 0 \}} \left(\sqrt{1+\tan^2\theta |Dv|^2} - \cos\theta 1_{\{v >0\}}\right)dx'\\
&= \int_{\{ u > 0\}} \left((1-\cos\theta) 1_{\{v >0\}} +\frac 12 \tan^2\theta |Dv|^2 + O(\theta^3) \right)dx'\\
&=\frac 12 \tan^2\theta \int_{ \{ u > 0 \} } \left(|Dv|^2 + 1_{\{v >0\}} +O(\theta)\right)dx' 
\end{split}
\end{equation}
Thus, as $\theta\to 0$, $v=\frac{u}{\tan\theta}$ should be an one-homogeneous minimizer of the functional
\begin{equation}\label{eq.AC.functional}
J(v) = \int_{\RR^n} \left(|Dv|^2 + 1_{\{v >0\}}\right)dx'.
\end{equation}
The functional \eqref{eq.AC.functional} is the well-known one-phase Alt--Caffarelli functional, and the above argument shows heuristically that $J$ should be the linearization of $\cA^\theta$ as $\theta \to 0$.

There is a vast literature concerning regularity of minimizers of the Alt--Caffarelli functional; see, for example, \cite{AltCaffarelli}, \cite{CaffarelliJerisonKening}, \cite{JerisonSavin}, \cite{Velichkov}. We also refer the readers to the book by Caffarelli--Salsa \cite{CaffarelliSalsa} for an excellent expository of the problem.  We note that the idea of $J$ being the linearization of $\cA^\theta$ is well-known to the experts (see e.g.\ \cite{CaffarelliFriedman,FeldmanInwon} and \cite[Section 5]{DipierroKarakhanyanValdinoci}), but this seems the first time that the fine approximation \eqref{eq.capillary.vs.AC.formally} has been rigorously established and used to deduce geometric consequences.

\begin{theo}\label{thm:graphical.convergence.to.AC.minimizer}
	Let $\theta_i\to 0+$, $\Omega_i$ be a sequence of cones minimizing $\cA^{\theta_i}$ with an isolated singularity at $0$, and $M_i=\partial \Omega_i\cap \RR^{n+1}_+$. Then for $i$ sufficiently large, $M_i$ is contained in the graph of a Lipschitz function $u_i$ over $\RR^{n}=\partial \RR^{n+1}_+$. Moreover, up to a subsequence, $\{\frac{u_i}{\tan\theta_i}\}$ converges in $(W^{1,2}_{loc}\cap C^\alpha)(\RR^n)$ to an one-homogeneous minimizer $v$ to the Alt--Caffarelli functional $J$ for all $\alpha\in (0,1)$, and the free-boundaries $\partial \{ u_i > 0 \} \to \partial \{ v > 0 \}$ in the local Hausdorff distance.
\end{theo}

Therefore a smooth (away from $0$) capillary cone with small angle has a fine approximation in by $1$-homogenous minimizers to $J$. Additionally, in low dimensions (e.g. $n\le 4$ by \cite{CaffarelliJerisonKening,JerisonSavin}), it is known that an one-homogeneous minimizer to \eqref{eq.AC.functional} is $(x\cdot \mathbf{n})_+$ for some unit vector $\mathbf n$. In this case, we show that the improved convergence 
\[\frac{u_i}{\tan\theta_i}\to v\]
holds actually in $C^2$ away from the origin. We use this to prove case (2) of Theorem \ref{thm:main.cone}. More generally, we have the following result.

\begin{theo}\label{thm:AC.bernstein.implies.capillary.bernstein}
	Let $n\ge 2$. Assume that the only one-homogeneous minimizers of $J$ on $\RR^n$ are given by $(x\cdot \mathbf{n})_+$ for some unit vector $\mathbf{n}$.  Then there exists $\theta_0=\theta_0(n)>0$, such that any minimizing cone for $\cA^\theta$ in $\RR^{n+1}_+$ with an isolated singularity at $0$ is flat.
\end{theo}

\subsection{Organization} In Section 2 we recall some preliminaries for capillary surfaces, minimal cones in the Euclidean space and the Alt-Caffarelli functional. In Section 3 we prove case (1) of Theorem \ref{thm:main.cone}. We carry out the rigorous relation between the capillary problem and the one-phase Bernoulli problem in Section 4, and prove case (2) of Theorem \ref{thm:main.cone}. Finally, we prove case (3) of Theorem \ref{thm:main.cone} in Section 5.

\subsection{Acknowledgements}  We would like to thank Daniela De Silva and William Feldman for bringing our attention to the connection between the capillary problem and the one-phase Bernoulli problem, and thank Guido De Philippis, Fang-Hua Lin, and Joaquim Serra for helpful conversations.

O.C. was supported by a Terman Fellowship and an NSF grant (DMS-2304432). N.E. was supported by an NSF grant (DMS-2204301).  C.L. was supported by an NSF grant (DMS-2202343) and a Simons Junior Faculty Fellowship.

\section{Preliminaries}

\subsection{Notation}

We take $(X,h)$ to be a smooth Riemannian manifold with boundary, $E\subset X$ an open set with $M=\partial E\subset \mathring{X}$ smooth. Let $\nu$ be the unit normal vector field of $M$ in $E$ that points into $E$, $\eta$ be the unit normal vector field of $\partial M$ in $M$ that points out of $X$, $\overline \eta$ be the unit normal vector field of $\partial X$ in  $X$ that points outward $X$, $\bar \nu$ be the unit normal vector field of $\partial M$ in $\partial X$ that points into $\partial E\cap \partial X$. Let $S: T_M\to T_M$ denote the shape operator defined by $S(Y) = -\nabla_Y \nu$, and $A$ the second fundamental form on $M$ given as $A(Y,Z) = \langle S(Y),Z\rangle = \langle \nabla_Y Z, \nu\rangle$. Particularly, if $M$ is unit sphere in $\RR^{n+1}$ viewed as the boundary of the unit ball $E$ then we have that $S=\id$ and $A =g_M$.

We write $\RR^{n+1}_+ = \{x\in \RR^{n+1}: x_1>0\}$, $\RR^n = \partial \RR^{n+1}_+= \{x_1=0\}$, and let $\pi: \RR^{n+1}_+\to \RR^{n}$ be the orthogonal projection. Given any set $S$, let $d(x,S)$ be the Euclidean distance from $x$ to $S$. We denote by $B_r(S) = \{x: d(x,S)<r\}$. Constants of the form $c$ or $c_i$ will always be $\geq 1$. 

\subsection{First and second variation of capillary functional}
Say $E$ is stationary for the capillary functional \eqref{eq:capillary.functional}, if $\frac{d}{dt}|_{t=0}\cA(E_t)\ge 0$ for all variations $E_t$. Take $M=\partial E\cap \mathring X$. The stationary condition is equivalent to
\[H_M = g \text{ in }\mathring X, \quad \langle \nu, \bar \eta \rangle = \cos\theta\text{ on }\partial M\cap \partial X.\]
$E$ is stable for \eqref{eq:capillary.functional}, if $\frac{d^2}{dt^2}|_{t=0}\cA(E_t)\ge 0$ for all variations $E_t$. The stability condition is equivalent to 
\begin{equation}\label{eq:stable}
\int_M \left(|\nabla_M f|^2 -(|A_M|^2+\Ric(\nu,\nu))f^2\right) d\cH^n -\cot\theta \int_{\partial M}A_M(\eta,\eta)f^2d\cH^{n-1}\ge 0,
\end{equation}
for all $f\in C_0^1(M)$. We refer the readers to \cite[Section 1]{RosSouam} for the deductions. Particularly, if $X=\RR^{n+1}$, then stability gives
\begin{equation}\label{eq:stable.euclidean}
\int_M (|\nabla_M f|^2 - |A_M|^2 f^2)d\cH^n -\cot\theta \int_{\partial M} A_M(\eta,\eta)f^2 d\cH^{n-1}\ge 0.
\end{equation}
We write $L_M = \Delta_M + |A_M|^2$ the Jacobi operator.

We will often work with capillary stable (minimizing) cones $\Omega$ in $\RR^{n+1}_+$. By definition, $\Omega\subset\RR^{n+1}_+$ is stationary and stable (minimizing) for $\cA$ in $\RR^{n+1}_+$. We say $\Omega$ is a \textit{smooth cone}, if $M=\partial \Omega\cap \RR^{n+1}_+$ is smooth away from the origin.

\subsection{The Alt-Caffarelli functional}

For an open set $U\subset \RR^n$ and $u\in W^{1,2}(U)$, write
\[J_U (u) = \int_U |Du|^2 + \chi_{U\cap \{u>0\}} dx\]
for the Alt--Caffarelli functional. We will say that $u \in W^{1,2}_{loc}(U)$ minimizes $J_U$ if for every $U' \subset\subset U$ and every $w - u \in W^{1,2}_0(U')$ we have $J_{U'}(u) \leq J_{U'}(w)$.  If $U = \RR^n$ we will simply write $J_{\RR^n} = J$, and call any minimizer of $J_{\RR^n}$ an entire minimizer of $J$.

We require the following basic fact.
\begin{lemm}\label{lemm.C0.bound.on.AC.minimizer}
	Let $u\in W_\textnormal{loc}^{1,2}(\RR^n)$ be an entire minimizer of $J$. Then there exists a constant $c_0(n)>0$ such that 
	\[\frac{1}{c_0}d(x,\partial\{u>0\}) \le u(x)\le d(x,\partial\{u>0\}).\]
\end{lemm}

\begin{proof}
	The lower bound follow from \cite[Corollary 3.6]{AltCaffarelli}, and the sharp upper bound is a consequence of the estimate $|Du|\le 1$ (see, e.g. \cite[Section 2.2]{JerisonSavin}).
\end{proof}

We will also make important use of the classification of entire minimizers.

\begin{theo}[\cite{JerisonSavin}]\label{theo.jerison.savin}
	Assume $n\le 4$. Let $u\in W^{1,2}_\textnormal{loc} (\RR^n)$ be an entire minimizer of $J$. Then $u=(x\cdot \textbf{n})_+$ for some unit vector $\textbf{n}$.
\end{theo}

\begin{proof}
	The classification of $1$-homogenous minimizers is due to \cite{JerisonSavin}.  The same classification holds for entire minimizers that are not a priori $1$-homogenous by an obvious blow-down argument and the Weiss monotonicity \cite{Weiss}.
\end{proof}

\section{The case $n=3$}

\begin{prop}\label{prop.stable.spectral}
	Suppose $\Omega\subset\RR^{n+1}_+$ is a cone, stationary and stable for $\cA^\theta$, and $M=\partial \Omega\cap\RR^{n+1}_+$. Assume that $M$ has an isolated singularity so that $\Sigma = M\cap S^{n}(1)$ is smooth. Then we have that
	\begin{equation}\label{eq:stable.spectral}
	\inf_{f\in C^1(\Sigma)} \frac{\int_\Sigma \left(|\nabla_\Sigma f|^2 - |A_\Sigma|^2f^2\right) - \cot\theta \int_{\partial\Sigma} A_\Sigma (\eta,\eta)f^2 }{\int_\Sigma f^2}\ge -\left(\frac{n-2}{2}\right)^2.
	\end{equation}
\end{prop}

\begin{proof}
	This follows from \cite[Lemma 4.5]{CaffarelliHardtSimon}. Note that any homogeneous extension of the first eigenfunction with the capillary boundary condition (i.e.\ $\frac{\partial f}{\partial \eta} = \cot\theta A(\eta,\eta)f$) satisfies the same boundary condition.
\end{proof}

\begin{proof}[Proof of Theorem \ref{thm:main.cone} when $n=3$]
	When $n=3$, \eqref{eq:stable.spectral} becomes
	\[\int_{\Sigma} (|\nabla_\Sigma f|^2 - |A_\Sigma|^2f^2 ) - \cot\theta \int_{\partial \Sigma} A_\Sigma (\eta,\eta)f^2 \ge -\frac 14 \int_\Sigma f^2,\]
	for every $f\in C^1(\Sigma)$. Setting $f=1$ and using the traced Gauss equation to write $|A_\Sigma|^2 = 1-K_\Sigma +\tfrac 12 |A_\Sigma|^2$, we have that
	\[\int_\Sigma K_\Sigma - \cot\theta \int_{\partial \Sigma} A_\Sigma(\eta,\eta) \ge \frac 34 \cH^2(\Sigma) +\frac 12 \int_{ \Sigma} |A_\Sigma|^2.\]
	Using \cite[(3.8)]{LiPolyhedron}, one finds that $\cot\theta A_\Sigma(\eta,\eta) = -k_g$, where $k_g$ is the geodesic curvature of $\partial \Sigma$ with respect to the outward unit normal vector field. Thus, the above gives
	\[\frac 34 \cH^2(\Sigma) \le \int_ \Sigma K_\Sigma + \int_{\partial \Sigma} k_g = 2\pi \chi(\Sigma).\]
	In particular, we have that $\chi(\Sigma)>0$ and thus $\Sigma$ is a topological disk. 
	
	We then proceed exactly as in \cite[Theorem 2.2]{FraserSchoen}. Precisely, let $u: D\to S^{n}_+$ be the proper minimal embedding such that $\Sigma = u(D)$. Take the standard complex coordinates $z=x_1+ix_2$, and consider the function
	\[\phi(z) = A_\Sigma\left(\frac{\partial}{\partial z},\frac{\partial}{\partial z}\right)^2.\]
	Here $A_\Sigma(\partial_z, \partial_z)$ is the second fundamental form in these coordinates, given by $A_\Sigma(\partial_z,\partial_z) = \langle \nabla_{du(\partial_z)} du(\partial_z), \nu\rangle$. By \cite[Theorem 2.2]{FraserSchoen}, we have that $\phi$ is a holomorphic function in $D$. Now consider the boundary behavior of $\phi$. For this, take $(r,\theta)$ the polar coordinates on $D$, so that $z=re^{i\theta}$. Set $w=\log z = \log r+i\theta$. Then $A(\partial_z,\partial_z) = \frac{1}{z^2}A(\partial_w,\partial_w)$. Along $\partial \Sigma$, we have that 
	\[du(\pa{}{r}) = \lambda \eta = \lambda(\cos\theta\bar \eta + \sin \theta \bar \nu),\quad \nu = \cos\theta \bar \nu - \sin\theta \bar \eta,\]
	
	here $\lambda = |du(\pa{}{r})|$. Consequently, we conclude that
	\[A_\Sigma \left(\pa{}{r},\pa{}{\theta}\right) = \lambda \left\langle \nabla_{du(\pa{}{\theta})} (\cos\theta \bar \eta +\sin \theta \bar \nu), \cos\theta \bar \nu - \sin\theta \bar \eta \right\rangle = 0,\]
	because $|\bar \nu|=|\bar \eta| =1$ and $\partial S^n_+$ is totally geodesic. Consequently, the same argument as in \cite[Theorem 2.2]{FraserSchoen} shows that $z^4 \phi(z)$ vanishes identically on $D$, and that $\Sigma$ is totally geodesic. 	 
\end{proof}

\section{The case $n=4$ with $\theta$ close to $0$}
In this section, we consider $\Omega \subset \RR^{n+1}_+$ that minimize the functional
\[\cA^\theta(\Omega) = |\partial^* \Omega \cap \RR^{n+1}_+| - \cos\theta |\partial^* \Omega\cap \partial_+\RR^{n+1}|\]
in $B_1$ or (later) in compact subsets of $\RR^{n+1}$.  Here $\theta\in (0,\frac{\pi}{2}]$.  We say $\Omega$ is a smooth minimizer of $\cA^\theta$ in $U$ if $M := \del \Omega \cap \R^{n+1}_+$ is a smooth hypersurface in $U$, which extends smoothly up to the solid boundary $\del \R^{n+1}_+$.

Most of our work in this Section is focused on proving the following convergence theorem, which is essentially a generalization/refinement of Theorem \ref{thm:graphical.convergence.to.AC.minimizer}.
\begin{theo}\label{thm:general-blow-up}
	Let $\theta_i\to 0+$, $\Omega_i$ be a sequence of minimizers of $\cA^{\theta_i}$ in $B_1$ which are smooth (resp. conical with an isolated singularity at $0$), and $M_i=\partial \Omega_i\cap \RR^{n+1}_+$. Then for a dimensional constant $\eps(n)$ and $i$ sufficiently large, $M_i \cap B_{1/2} \cap B_\eps(\R^n)$ (resp. $M_i \cap B_{1/2}$) is contained in the graph of a Lipschitz function $u_i$ over $B_{1/2}^n \equiv \partial \RR^{n+1}_+ \cap B_{1/2}$. Moreover, up to a subsequence, $\theta_i^{-1} u_i$ converges in $(W^{1,2}_{loc}\cap C^\alpha_{loc})(B_{1/4}^n)$ to a minimizer $v$ (resp. one-homogenous minimizer $v$) to the Alt--Caffarelli functional $J$ for all $\alpha\in (0,1)$, and the free-boundaries $\partial \{ u_i > 0 \} \to \partial \{ v > 0 \}$ in the local Hausdorff distance in $B_{1/4}^n$.
	
	If $n \leq 4$, then $v$ is entirely regular (resp. linear), and the convergence $\theta_i^{-1} u_i \to v$ is $C^{2,\alpha}_{loc}$ in $B_{1/4}^n$.  Near the free-boundary this is interpreted in the sense of the Hodograph transform.
\end{theo} 

\begin{rema}\label{rem:curvature-est}
	More generally, in Theorem \ref{thm:general-blow-up} (and Lemma \ref{lem:curvature-est}) instead of assuming $n \leq 4$ one can assume $n + 1 \leq k^*$, where $k^*$ is the smallest dimension in which there is a non-linear $1$-homogenous minimizer of $J$ in $\R^{k^*}$.  The value of $k^*$ is not yet known, but by the works of \cite{JerisonSavin}, \cite{DeSilvaJerison}, $k^* \in \{ 5, 6, 7 \}$.
\end{rema}

\subsection{Preliminaries}

Our first Lemma says that if the interior surface $M := \partial\Omega \cap \RR^{n+1}_+$ is smooth (up to the boundary), then in a region uniform far away from the interface $\partial M \cap B_1 \subset \partial \RR^{n+1}$, $M$ is mass-minimizing in the sense of boundaries.

\begin{lemm}\label{lemm:min.boundary}
	Let $\Omega$ be a smooth minimizer of $\cA^\theta$ in $B_1$, and let $M = \partial \Omega \cap B_1 \cap \RR^{n+1}_+$, $E = \Omega \cup (\pi^{-1}(\partial \Omega \cap \partial \RR^{n+1}_+) \setminus \RR^{n+1}_+ )$, and $S = \pi^{-1}(\del M \cap \partial \RR^{n+1}_+ \cap B_1)$.  Then	\[\partial [E] \llcorner \left(B_1 \setminus S \right) \equiv [M ] \]
	is a mass-minimizing boundary in $B_1\setminus S$.
\end{lemm}

\begin{proof}
	By construction we have $\partial E = M$ in $B_1 \setminus S$.  Let $F$ be a Caccioppoli set in $B_1$ such that $F\Delta E$ is contained in a compact subset $W$ of $B_1\setminus S$. Without loss of generality (by replacing $W$ by a compact set containing it), we can assume that $W=\pi^{-1}(W_0)\cap B_r$ for $W_0\subset \RR^n$ and $r<1$.  Since $M \equiv \partial\Omega \cap B_1 \cap \RR^{n+1}_+$ is a smooth minimal hypersurface and (by the maximum principle) meets $\partial \RR^{n+1}_+$ only at $\partial M$, there is an $\eps > 0$ so that
	\begin{equation}\label{eqn:min.boundary-1}
	d(x, \partial \RR^{n+1}_+) \geq 2\eps \quad \forall x \in \partial\Omega \cap W.
	\end{equation}
	
	Define the piecewise linear map $p_\eps(x_1, \ldots, x_{n+1}) = (\max \{ x_1, \eps\}, x_2, \ldots, x_{n+1})$, and define $F_\eps = F \cup (\{ x_1 > \eps\} \cap B_1)$.  We have $\partial^* F_\eps = p_\eps(\partial^* F)$ and $p_\eps(\partial^* F) \cap W = p_\eps(\partial^*F \cap W)$, and so since $\Lip(p_\eps) \leq 1$ we get the inequality $|\partial^* F_\eps \cap W| \leq |\partial^* F \cap W|$.
	From \eqref{eqn:min.boundary-1} we have $F_\eps \Delta \Omega \subset\subset W \cap \{ x_1 > 0 \}$.  Therefore if we set $\Omega' = F_\eps \cap \RR^{n+1}_+$ we can make the comparison $\cA_{W}^\theta(\Omega) \leq \cA_{W}^\theta(\Omega')$.  But since $\partial^*\Omega' \cap \partial\RR^{n+1}_+ = \partial^*\Omega \cap \partial\RR^{n+1}_+$, we can deduce the inequalities
	\[
	|\partial^*E \cap W| = |\partial^*\Omega \cap \RR^{n+1}_+ \cap W| \leq |\partial^* \Omega' \cap \RR^{n+1}_+ \cap W| = |\partial^*F_\eps \cap W| \leq |\partial^* F \cap W|. \qedhere
	\]
\end{proof} 

We next recall the following application of the Allard regularity theorem.

\begin{theo}\label{thm:allard}
	There is an $\eps_1(n)$ such that the following holds. Let $T=\partial [E]$ be a mass-minimizing boundary in $B_r$ such that
	\[  0\in \spt T, \quad \sup_{\spt T\cap B_r} r^{-1} |x_1|<\eps\le \eps_1.\]
	Then $\spt T\cap B_{r/2}=\graph_{\RR^n}(u)$ for some $C^2$ function $u$ with the estimate \[r|D^2u| + |Du| + r^{-1}|u|\le c_1(n)\eps , \quad\text{in }B_{r/2}.\]
\end{theo}

\begin{proof}
	The estimate is scaling invariant so we assume $r=1$ without loss of generality.  Assume the Proposition fails: then for any predetermined $c(n)$, there exists a sequence of mass-minimizing boundaries $T_j=\partial[E_j]$ and numbers $\eps_j\to 0$, so that $0 \in \spt T_j$ and $\sup_{ \spt T_j \cap B_1 } |x_1| \leq \eps_j$, but $\spt T_j \cap B_{1/2}$ is not the graph of a $C^2$ function over $\{ x_1 = 0 \}$ satisfying $|u|_{C^2} \leq c(n) \eps_j$.
	
	Passing to a subsequence, by compactness of mass-minimizing boundaries we can assume $E_j \to E$ and $T_j \to T = \partial [E]$ for some mass-minimizing boundary $T$.  However our contradiction hypothesis implies $0 \in \spt T$ and $\spt T \subset \{ x_1 = 0\}$.  Therefore we must have $T = \pm \partial [\{x_1 < 0\}]$, and hence $T_j \to [\{ x_1 = 0 \}]$ with multiplicity-one.  Allard's them implies that for $j \gg 1$, $\spt T_j \cap B_{1/2} = \graph_{\RR^n}(u)$ with $|u|_{C^2} \leq c(n) \sup_{\spt T_j \cap B_{3/4}} |x_1| \leq c(n) \eps_j$, which is a contradiction.
\end{proof}

We recall the Harnack inequality for harmonic functions on mass-minimizing boundaries by Bombieri--Guisti \cite{BombieriGuisti}.

\begin{theo}[{\cite[Theorem 5]{BombieriGuisti}}]\label{thm:harnack}
	Let $T=\partial [E]$ be a mass-minimizing boundary in $B_1^{n+1}$. There are constants $\sigma(n)\in (0,1)$, $c_2(n)$ such that if $u\in C^1(B_1)$ satisfies 
	\[\int \nabla u\cdot \nabla \phi d\mu_T=0,\quad \forall \phi\in C_0^1(B_1), u>0 \text{ on }\spt T,\]
	then 
	\[\sup_{\spt T\cap B_{\sigma} } u\le c_2 \inf_{ \spt T\cap B_{\sigma}} u.\]
\end{theo}

\begin{rema}\label{rema:DSS}
	It would be interesting to prove a version of Theorem \ref{thm:harnack} on the reduced boundary of a capillary minimizer, which holds up to the boundary (cf.\ \cite{DeSilvaSavin}). This would simplify the arguments below.
\end{rema}

Our final preliminary lemma rules out a minimizer $\Omega$ of $\cA^\theta$ having large pieces that stay too close to the boundary.
\begin{lemm}\label{lem:flat-is-zero}
	There is $\eps_3(n) > 0$ so that the following holds.  Let $\Omega \subset \RR^{n+1}_+$ be a minimizer for $\cA^\theta$ in $B_1$, and $u : B_1^n \to \R$ a non-negative Lipschitz function.  Suppose either: $\Omega \cap B_1 = \{ x \in B_1 : 0 < x_1 < u(\pi(x)) \}$ and $\sup_{B_{1/2}^n} u \leq \eps_3 \theta$; or $\Omega \cap B_1 = \{ x \in B_1 : u(\pi(x)) < x_1 \}$ and $\sup_{B_{1/2}^n} u \leq \eps_3$.  Then $u \equiv 0$ in $B_{1/4}^n$.
\end{lemm}

\begin{proof}
	By allowing for $\theta \in (0, \pi)$ in the below proof, and replacing $\Omega$ with $B_1 \setminus \Omega$, $\theta$ with $\pi - \theta$ as necesary, there will be no loss of generality in assuming we are in the first case.   We loosely follow the ideas of \cite[Lemma 3.4]{AltCaffarelli}.  Let $\phi(r = |x|)$ be the graphing function of the upper half of the standard catenoid, defined on $\{ x \in \R^n : |x| > 1 \}$.  Standard estimates (e.g. \cite{SchoenUniqueness}) imply that as $r \to \infty$, 
	\[
	\phi(r) = \left\{ \begin{array}{l l} a \log(r) + O(1) & n = 2 \\ b_n - d_n r^{2-n} + O(r^{1-n}) & n \geq 3. \end{array} \right. ,
	\]
	for some constants $a, b_n, d_n > 0$.  Choose $\lambda(n, \eps, \theta), \mu(n, \eps, \theta) > 0$ (and ensure $\eps(n)$ is small) so that if
	\[
	v = \max\{ \lambda \phi(|x|/\lambda) - \mu, 0\},
	\]
	then $v(r = 1/4) = 0$ and $v(r = 1/2) = 2\eps \theta$.  We note it then follows from our choice of $\lambda$, the structure of $\phi$ and the smallness of $\eps(n)$ that $|\partial_r v|_{r = 1/4}| \leq c(n) \eps \theta$.
	
	The graph of $v$ is a scaled and translated portion of the catenoid.  Write $\nu$ for the unit normal vector of $\graph_{\RR^n}(v)$ such that $\nu\cdot e_1>0$, and extend $\nu$ to be defined on all of $\RR^{n+1}$ so it is constant in the $x_1$ direction. A standard computation shows that in $\{ x \in \RR^{n+1} : |\pi(x)| \geq 1/4 \}$, the $n$-form $\omega := \star \nu$ is a calibration for $\graph_{\RR^n}(v)$. Let us define the hypersurface
	\[
	S = \{ x = (x_1, x') \in \RR \times \RR^n : |x'| = 1/4 \text{ and } 0 < x_1 < u(x') \},
	\]
	endowed with the inward-pointing orientation, and define the cylinder
	\[
	U = \{ x = (x_1, x') \in \RR \times \RR^n : u(x') > v(x')\text{ and } x' \in B_{1/2}^n \setminus B_{1/4}^n \},
	\]
	so that, because $v > u$ on $\partial B_{1/2}^n$,
	\[
	\partial U = [\graph_{\RR^n}(u) \cap U] - [\graph_{\RR^n}(v) \cap U] + S
	\]
	(here we equip the graphs with an orientation so its unit normal vectors point into positive $x_1$ direction). Now since $\min(u, v)$ is a Lipschitz function which agrees with $u$ along $\partial B_{1/2}^n$, we can make the comparison
	\[
	\cA^\theta(\Omega) \leq \cA^\theta(\{ 0 < x_1 < \min ( u(\pi(x)), v(\pi(x)) ) \})
	\]
	and thereby compute:
	\begin{align}
	&\int_{B_{1/4}^n \cap \{ u > 0 \}} \sqrt{ 1 + |Du|^2} - \cos \theta dx
	\leq \int_{ \{ u > 0 \} \cap B_{1/2}^n \setminus B_{1/4}^n} \sqrt{1 + |D\min(u, v)|^2} - \sqrt{1 + |Du|^2} dx \nonumber 
	\\
	&= \mathcal{H}^n( \graph_{\RR^n}(v) \cap U) - \mathcal{H}^n(\graph_{\RR^n}(u) \cap U) \nonumber 
	\leq \int_{\graph_{\RR^n}(v) \cap U} \omega - \int_{\graph_{\RR^n}(u) \cap U} \omega \nonumber  
	\\
	&\hspace{2cm}= \int_S \omega  = \int_{\partial B_{1/4}^n} \frac{\partial_r v}{\sqrt{1+|D v|^2}} u 
	\leq c(n) \eps \theta \int_{\partial B_{1/4}^n} u. \label{eqn:flat-is-zero-1}
	\end{align} 
	
	On the other hand, we claim we have the estimate
	\begin{align}\label{eqn:flat-is-zero-2}
	\int_{\partial B_{1/4}^n} u \leq c(n) \theta^{-1} \int_{B_{1/4}^n \cap \{ u > 0 \}} \sqrt{1+|D u|^2} - \cos\theta dx.
	\end{align}
	Combining \eqref{eqn:flat-is-zero-1} with \eqref{eqn:flat-is-zero-2} will give $\int_{\partial B_{1/4}^n} u = 0$ provided $\eps(n)$ is sufficiently small, which by a trivial comparison argument implies $u = 0$ on $B_{1/4}$.

	To prove \eqref{eqn:flat-is-zero-2}, we break into two cases.  For $\theta \in [0, \pi/2]$, we first note that we have the inequalities
	\begin{align*}
	2\sqrt{1+|Du|^2} - |Du| \theta - 1 &\geq \sqrt{4 - \theta^2} - 1 \\
	&\geq 1-\theta^2/3 \\
	&\geq 1-\theta^2/2+\theta^4/24 \\
	&\geq \cos\theta
	\end{align*}
	The first inequality follows by considering the minimum of the function $t \mapsto 2 \sqrt{1+t^2}- t \theta - 1$ on $\{ t \geq 0 \}$; the second and the third hold for $\theta \in [0, \frac{\pi}{2}]$ by elementary algebra; the fourth holds by the Taylor remainder theorem.
	
	Now using the trace inequality, then our bound $|u| \leq \eps \theta$, and then the above inequality, we can estimate
	\begin{align*}
	\int_{\partial B_{1/4}^n} u 
	&\leq c(n) \int_{B_{1/4}^n} (|D u| + u )dx \\
	&\leq c(n) \theta^{-1} \int_{\{ u > 0 \} \cap B_{1/4}^n} (|D u| \theta + 1- \cos\theta)  dx \\
	&\leq 2c(n) \theta^{-1} \int_{\{ u > 0 \}  \cap B_{1/4}^n} \left(\sqrt{1 + |Du|^2} - \cos\theta \right) dx.
	\end{align*}
	This proves the assertion when $\theta \in [0, \pi/2]$.  When $\theta \in [\pi/2, \pi]$, then we can obtain the last two inequalities by the fact $\cos\theta \leq 0$ and the inequality $t + 1 \leq 2 \sqrt{t^2 + 1}$.  This proves the Lemma.
\end{proof}

\subsection{Inhomogenous blow-up to Alt-Caffarelli}

We show here that capillary minimizers of small angle with good a priori estimates near the boundary are graphical, and that the graphing function looks close to a minimizer of the Alt-Caffarelli energy at the scale of the angle.  We first prove that non-degeneracy estimates \eqref{eqn:graph-prop-hyp} near the boundary can be extended to graphical estimates at both near and far from the boundary.  In Section \ref{sec:a priori}, we shall prove using a bootstrap-type argument that estimates like \eqref{eqn:graph-prop-hyp} hold in a neighborhood of uniform size.

\begin{lemm}[Graphical propogation]\label{lem:graph-prop}
	Given $\gamma > 0$, there exists $\theta_0, \eps, c$ positive and depending only on $(n, \gamma)$ so that the following holds. Suppose $\Omega\subset \R^{n+1}_+$ is a smooth minimizer for $\cA^{\theta}$ in $B_1$ for some $\theta \leq \theta_0$, and let $M=\partial \Omega\cap \R^{n+1}_+$. \emph{Assume} that $0 \in \del M$, and
	\begin{equation}\label{eqn:graph-prop-hyp}
	\frac{1}{2c_0}\tan \theta \, d(\pi(x), \partial M)\le x_1\le 2\tan\theta\,  d(\pi(x),\partial M),
	\end{equation}
	for all $x \in M \cap B_\gamma(\del M)$, where $c_0$ as in Lemma \ref{lemm.C0.bound.on.AC.minimizer}.
	
	Then there exists a Lipschitz function $u: B_{1/2}^n\to \R$ and so that the following holds:
	\begin{enumerate}
		\item $\Omega = \{ x : 0 < x_1 < u(\pi(x)) \}$ $\leb^{n+1}$-a.e. in $\{ 0 < x_1 < \eps\} \cap B_{1/2}$; \label{item:graph-prop-1}
		\item $M = \graph_{\R^n}(u)$ in $\{ 0 < x_1 < \eps \} \cap B_{1/2}$; \label{item:graph-prop-2}
		\item $\del M = \del \{ u > 0 \}$ in $B_{1/2}^n$; \label{item:graph-prop-3}
		\item $\mathrm{Lip}(u)\le c\tan\theta$. \label{item:graph-prop-4}
		\item for every $z\in B_{1/2}^n \cap \{ u > 0 \}$, we have \label{item:graph-prop-5}
		\begin{equation}\label{eqn:non-deg}\frac{1}{c}\tan\theta \, d(z,\partial \{u>0\})\le u(z)\le c\tan\theta \, d(z,\partial \{u>0\}).\end{equation}
	\end{enumerate}
\end{lemm}

\begin{rema}
	If $\Omega$ is smooth, then assumptions \eqref{eqn:graph-prop-hyp} will always hold for some $\gamma$.
\end{rema}

\begin{proof}
	Let $\eps_1$, $c_1$ be the constants from Theorem \ref{thm:allard}, $\sigma$, $c_2$ be the constants from Theorem \ref{thm:harnack}, and $\eps_3$ the constant from Lemma \ref{lem:flat-is-zero}.  Define
	\begin{equation}\label{eqn:graph-prop-1}
	\rho = \frac{\sigma\gamma}{8}, \quad K =  \frac{c_0 c_2^{\lceil 1/\rho \rceil}}{\rho}, \quad \eps = \frac{4\eps_1}{K}, \quad \tan \theta_0 = \frac{\rho}{100 c_1 c_2 K} \min \{ \eps, \eps_1, \eps_3 \}.
	\end{equation}
	Here $\lceil 1/\rho \rceil$ is the smallest integer that is larger than or equal to $1/\rho$. We claim that
	\begin{equation}\label{eqn:graph-prop-2}
	M \cap \{ 0 < x_1 < \eps \} \cap B_{3/4} \setminus B_{\gamma/4}(\del M) \subset \{ \tan\theta/K < x_1 < K \tan\theta \}.
	\end{equation}
	
	By our choice of constants, and since $0 \in \del M$, to prove \eqref{eqn:graph-prop-2}, it will suffice to prove by induction that for $N = 0, 1, 2, \ldots, \lceil 1/\rho\rceil$, we have the inclusion
	\begin{align}
	&M \cap \{ \gamma/4 < d(x, \del M) < \gamma/2 + N\rho \} \cap \{ 0 < x_1 < 4\eps_1\rho/c_2^N \} \cap B_{3/4} \nonumber \\
	& \quad \subset \{ (\gamma/8 c_0) \tan\theta / c_2^N < x_1 < 2\gamma c_2^N \tan \theta \}. \label{eqn:graph-prop-3}
	\end{align}
	
	When $N = 0$ then \eqref{eqn:graph-prop-3} trivially holds by our hypothesis \eqref{eqn:graph-prop-hyp}.  Let $N \geq 1$, and suppose by inductive hypothesis \eqref{eqn:graph-prop-3} holds with $N-1$ in place of $N$.  Take $y \in \R^n \cap \del B_{\gamma/2 + (N-1)\rho}(\del M) \cap B_{3/4}$, and consider the ball $B_{4\rho}(y)$.  In $B_{4\rho/\sigma}(y)$, $M$ is a minimizing boundary.  If $M \cap B_{4\rho}(y) \cap B_{4\eps_1 \rho/c_2^N}(\R^n) \neq \emptyset$, then by Theorem \ref{thm:harnack} (with $u(x) = x_1$) we must have $M \cap B_{4\rho}(y) \subset B_{4\eps_1\rho/c_2^{N-1}}(\R^n)$, and hence by Theorem \ref{thm:allard} $M \cap B_{2\rho}(y) = \graph_{\R^n}(u)$ for some smooth function $u$ satisfying $|u| \leq 4\eps_1 \rho/c_2^{N-1}$.
	
	In particular, there is a point $z \in M \cap B_{2\rho}(y) \cap B_{\gamma/2 + (N-1)\rho}(\del M) \cap B_{4\eps_1 \rho/c_2^{N-1}}(\R^n)$, which by our inductive hypothesis satisfies $(\gamma/8c_0)\tan\theta/c_2^{N-1} < z_1 < 2\gamma c_2^{N-1} \tan\theta$.  Using Theorem \ref{thm:harnack} again, we deduce $M \cap B_{2\rho}(y) \cap B_{4\eps_2\rho/c_2^N}(\R^n) \subset \{ (\gamma/8c_0) \tan\theta/c_2^N < x_1 < 2\gamma c_2^N \tan\theta \}$.  Repeating this argument for every admissible $y$ proves \eqref{eqn:graph-prop-3}, and hence \eqref{eqn:graph-prop-2} follows by induction.

	Let $M'= M\cap \{0<x_1 <\eps\}$. Given $y = (y_1, y') \in M' \cap B_{3/4}$ take $r = \min\{ d(y', \del M)/16, \rho \}$, and note that by \eqref{eqn:graph-prop-hyp}, \eqref{eqn:graph-prop-2}, \eqref{eqn:graph-prop-1} we have
	\begin{equation}\label{eqn:graph-prop-4}
	y_1 \leq (K/\rho)  \tan\theta r \leq \frac{\min\{ \eps_1, \eps_3 \}}{100 c_1c_2} r \leq r/100
	\end{equation}
	so that in particular we have $y \in M' \cap B_{r/4}(y')$.  By Theorem \ref{thm:harnack}, we have $M' \cap B_{4r}(y') \subset \{ x_1 <  c_2 (K/\rho) \tan \theta r \}$, and so by \eqref{eqn:graph-prop-4} and Theorem \ref{thm:allard} $M' \cap B_{2r}(y') = \graph_{\R^n}(u)$ for $u : B_r^n(y') \to \R$ a smooth function satisfying
	\begin{equation}\label{eqn:graph-prop-5}
	r^{-1} |u| + |Du| \leq c_1 c_2 (K/\rho) \tan \theta.
	\end{equation}
	By our choice of $\theta_0$, Lemma \ref{lem:flat-is-zero} then implies that
	\begin{equation}\label{eqn:graph-prop-6}
	\Omega \cap B_{r/2}(y') = \{ x \in B_{r/2}(y') : 0 < x_1 < u(\pi(x)) \} .
	\end{equation}
	
	From \eqref{eqn:graph-prop-5}, we deduce that $\pi(M') \cap B_{1/2}^n$ is an open subset of $\R^n$, and $M' \cap B_{1/2} = \graph_{\R^n}(u) \cap B_{1/2}$ for $u : \pi(M') \cap B_{1/2}^n \to \R$ a smooth positive function satisfying
	\begin{equation}\label{eqn:graph-prop-7}
	d(x, \del M) \tan\theta/c \leq u(x) \leq c d(x, \del M) \tan\theta, \quad |Du(x)| \leq c \tan\theta
	\end{equation}
	for all $x \in \pi(M') \cap B_{1/2}^n$, with $c = c(n, \gamma)$.  The upper bounds in \eqref{eqn:graph-prop-7} follow from \eqref{eqn:graph-prop-5}, and the lower bound from \eqref{eqn:graph-prop-hyp}, \eqref{eqn:graph-prop-2}.  By \eqref{eqn:graph-prop-5} and our choice of $\theta_0$, we have
	\begin{equation*}
	\del \pi(M') \cap B_{1/2}^n = \del M \cap B_{1/2}^n,
	\end{equation*}
	and so we can extend $u$ by zero to all of $B_{1/2}^n$ to obtain a Lipschitz function satisfying the required properties (\ref{item:graph-prop-2}), (\ref{item:graph-prop-3}), (\ref{item:graph-prop-4}), (\ref{item:graph-prop-5}).
	
	Finally, we observe that if $\Omega'$ is any connected component of $\Omega \cap \{ 0 < x_1 < \eps \}$, then by \eqref{eqn:graph-prop-hyp} necessarily $\del \Omega' \cap \{ 0 < x_1 < \eps \} \cap B_{3/4}$ is a non-empty subset of $M' \cap B_{3/4}$.  From this and \eqref{eqn:graph-prop-6} we obtain property (\ref{item:graph-prop-1}).
\end{proof}

By a trivial modification of the above proof, we also obtain a version of the graphical propogation for smooth cones.
\begin{lemm}[Graphical propogation for cones]\label{lem:graph-prop-cones}
	Given $\gamma > 0$, there exist $\theta_0(n, \gamma)$, $c(n, \gamma)$ so that the following holds.  Suppose $\theta \leq \theta_0$, and $\Omega \subset \R^{n+1}_+$ is a dilation-invariant minimizer for $\cA^\theta$ which is smooth away from the cone point $0$.  Write $M = \del \Omega \cap \R^{n+1}_+$.  Assume that
	\begin{equation}\label{eqn:graph-prop-cones-hyp}
	\frac{1}{2c_0} \tan\theta d(\pi(x), \del M) \leq x_1 \leq 2 \tan\theta d(\pi(x), \del M)
	\end{equation}
	for all $x \in M \cap B_\gamma(\del M) \cap \del B_1$, where $c_0$ as in Lemma \ref{lemm.C0.bound.on.AC.minimizer}.
	
	Then there exists a $1$-homogenous Lipschitz function $u : \R^n \to \R$ so that:
	\begin{enumerate}
		\item $\Omega = \{ x : 0 < x_1 < u(\pi(x)) \}$ $\leb^{n+1}$-a.e.;
		\item $M = \graph_{\R^n}(u)$ in $\R^{n+1}_+$;
		\item $\del M = \del \{ u > 0 \}$;
		\item $\Lip(u) \leq c \tan\theta$;
		\item for every $z \in \{ u > 0 \}$ we have
		\[
		\frac{1}{c} \tan\theta d(z, \del \{ u > 0 \}) \leq u(z) \leq c \tan\theta d(z, \del \{ u > 0 \}).
		\]
	\end{enumerate}
\end{lemm}

\begin{proof}
	We first note that by the maximum principle in the sphere $M \cap \del B_1$ is necessarily connected, and by the maximum principle for stationary varifold cones we must have $\del M \cap \del B_1 \neq \emptyset$.  By an essentially verbatim argument as Lemma \ref{lem:graph-prop}, we can therefore find a $1$-homogenous (because $M$ is conical) Lipschitz function $u: B_2^n \setminus B_{1/2}^n \to \R$ satisfying properties (\ref{item:graph-prop-1})-(\ref{item:graph-prop-5}) of Lemma \ref{lem:graph-prop} (except with $B_{1/2}^n$ replaced with $B_2^n \setminus B_{1/2}^n$).
	
	Ensuring $\theta(n, \gamma)$ is sufficiently small so that $|u| < \eps$ on $B_2^n \setminus B_{1/2}^n$, property (\ref{item:graph-prop-1}) and connectivity of $M$ implies that in fact (\ref{item:graph-prop-1}), (\ref{item:graph-prop-2}) hold without the restriction on $x_1$.  If we extend $u$ to $\R^n$ by $1$-homgeneity, then (\ref{item:graph-prop-1})-(\ref{item:graph-prop-5}) will continue to hold on $\R^n$, as the estimates in (\ref{item:graph-prop-4}), (\ref{item:graph-prop-5}) are invariant under $1$-homogenous rescaling.
\end{proof}

As $\theta \to 0$, the graphing function of Lemmas \ref{lem:graph-prop}, \ref{lem:graph-prop-cones} resembles more closely a minimizer of the Alt-Caffarelli functional.  This is made precise below.
\begin{prop}[Inhomogenous blow-up]\label{prop:blow-up}
	Let $\theta_\mu>0$ be a sequence approaching $0$, $\Omega_\mu\subset\R^{n+1}_+$ a sequence of (not necessarily smooth!) minimizers for $\cA^{\theta_\mu}$ in $B_1$, and $M_\mu=\partial \Omega_\mu \cap \R^{n+1}_+$. Assume there are fixed $\Gamma, \eps > 0$ and Lipschitz functions $u_\mu : B_1^n \to \R$ satisfying
	\begin{enumerate}
		\item $\Omega_\mu = \{ x : 0 < x_1 < u_\mu(\pi(x)) \}$ in $B_\eps(\R^n) \cap B_1$,
		\item $M_\mu = \graph_{\R^n}(u_\mu)$ in $\{ 0 < x_1 < \eps \} \cap B_1$ , 
		\item $0 \in \del M_\mu = \del \{ u_\mu > 0 \}$ in $\R^n \cap B_1$,  
		\item $\Lip(u_\mu) \leq \Gamma \tan\theta_i$,  
		\item for every $z \in B_1^n \cap \{ u_\mu > 0 \}$, we have the estimate
		\[
		\frac{1}{\Gamma} \tan\theta_\mu d(z, \del \{ u_\mu > 0 \}) \leq u_i(z) \leq \Gamma \tan\theta_\mu d(z, \del \{ u_\mu > 0 \}).
		\]
	\end{enumerate}
	Then up to a subsequence,
	\[v_\mu = \frac{1}{\tan\theta_\mu}u_\mu\]
	converges in $(W^{1,2}_{loc}\cap C^\alpha_{loc})(B_{1/2}^n)$, for all $\alpha<1$, to a Lipschitz function $v : B_{1/2}^n \to \R$ that minimizes the Alt--Caffarelli functional $J$ in $B_{1/2}^n$, and the free-boundaries $\partial \{ v_\mu > 0 \} \to \partial \{ v > 0 \}$ in the local Hausdorff distance in $B_{1/2}^n$.
	
	If, \emph{additionally}, one assumes the $M_\mu$ are smooth and admit an upper bound of the form
	\begin{equation}\label{eqn:blow-up-hypA}
	\sup_\mu \sup_{B_{1}} \theta_\mu^{-1} |A_{M_\mu}| < \infty,
	\end{equation}
	then $\del \{ v > 0 \} \cap B_{1/2}$ is entirely regular, and the convergence $v_\mu \to v$ is $C^{2,\alpha}_{loc}(B_{1/2}^n)$ in the sense of Hodograph transforms.  In particular, we get
	\begin{equation}\label{eqn:blow-up-conclA}
	\sup_{B_r} \theta_\mu^{-1} |A_{M_\mu}| \to \sup_{B_r^n \cap \{ v > 0 \}} |D^2 v| \quad \forall r < 1/2.
	\end{equation}
\end{prop}

\begin{rema}
	We shall show in Lemmas \ref{lem:bootstrap}, \ref{lem:graph-prop} that hypotheses (1)-(5) will always hold for any sequence of \emph{smooth} minimizers with $\theta_\mu \to 0$.  Additionally, we shall show in Lemma \ref{lem:curvature-est} that \eqref{eqn:blow-up-hypA} always holds for a sequence of smooth minimizers of dimension $n \leq 4$ (or more generally $n + 1\leq k^*$, see Remark \ref{rem:curvature-est}).
\end{rema}

\begin{proof}
	We first prove the $C^\alpha_{loc} \cap W^{1,2}_{loc}$ and Hausdorff convergence for $u_\mu$ satisfying only properties (1)-(5).  By our hypothesis, we know that the functions $v_\mu$ are uniformly bounded in $W^{1,\infty}(B_{1})$, and so passing to a subsequence we have $v_\mu \to v$ in $C^\alpha_{loc}(B_{1}^n)$ for all $\alpha < 1$, for some $v \in W^{1,\infty}(B_{1})$ satisfying $|v| + \Lip(v) \leq c(n, \gamma)$.	
	
	We claim that $\partial \{ v_\mu > 0 \} \to \partial \{ v > 0 \}$ in the local Hausdorff distance in $B_{1/2}$, which will imply among other things that $1_{\{ v_\mu > 0 \}} \to 1_{\{ v > 0 \}}$ in $L^1_{loc}(B_{1/2})$ also.  If $x \in \partial \{ v > 0 \} \cap B_{1/2}$, then for any $\eps > 0$ small, choose $z \in B_\eps(x)$ with $0 < v(z) < \eps$.  For $\mu \gg 1$, we must have $0 < v_\mu(z) < 2 \eps$, and hence by \eqref{eqn:non-deg} we have $d(x, \partial \{ v_\mu > 0 \}) \leq c(n, \gamma) \eps$.  Thus we obtain a sequence $x_\mu \in \partial \{ v_\mu > 0 \}$ with $x_\mu \to x$.  Conversely, suppose we have $x_\mu \in \partial \{ v_\mu > 0 \}$ with $x_\mu \to x \in B_{1/2}$.  If $x \not\in \partial \{ v > 0 \}$, then there is a ball $B_{2r}(x)$ on which $v = 0$.  But then since $v_i \to 0$ on $B_r(x)$, Lemma \ref{lem:flat-is-zero} implies $v_\mu \equiv 0$ on $B_{r/2}(x)$ for $\mu \gg 1$, contradicing our choice of $x_\mu$.  So we must have $x \in \partial \{ v > 0 \}$.  This proves our claim.   
	
	For any fixed $B_{2r}(p) \subset \{ v > 0 \}$, then $B_{2r}(p) \subset \{ v_\mu > 0 \}$ for all $\mu \gg 1$.  Since each $u_\mu$ solves the minimal surface equation in $B_{2r}(p)$, by standard elliptic estimates $v_\mu \to v$ smoothly in $B_r(p)$.  So $v_\mu \to v$ smoothly on compact subsets of $B_{1/2} \cap \{ v > 0 \}$.  Together with the uniform Lipschitz bounds on $v_\mu$ and the local Hausdorff convergence $\partial \{ v_\mu > 0 \} \to \partial \{ v > 0 \}$, we get $v_\mu \to v$ in $W^{1,2}_{loc}(B_{1/2})$.
	
	We now prove that $v$ minimizes $J$ in $B_{1/2}$. For this, let $r< 1/2$ and let $w$ be a Lipschitz function satisfying $\spt (w - v) \subset B_r$.  Choose a smooth function $\eta : B_{1/2} \to [0, 1]$ with $\spt (w - v) \Subset \spt \eta \subset B_{r}$.  Define the functions $w_\mu = w + (1-\eta)(v_\mu - v)$, and the domains
	\[
	\Omega'_\mu = \{ x \in B_{1/2}^{n+1} : 0 < x_1 < \max ( 0, \tan\theta_\mu w_\mu ) \}.
	\]
	Then $\Omega_\mu \Delta \Omega_\mu' \Subset B_{1/2}^{n+1}$ for $\mu \gg 1$, and so
	\begin{align*}
	&\cA^{\theta_\mu}(\Omega_\mu) \leq \cA^{\theta_\mu}(\Omega_\mu') \\
	&\implies \int_{\{u_\mu > 0 \} \cap B_r} \sqrt{ 1 + |Du_\mu|^2} - \cos\theta_\mu 1_{\{u_\mu > 0 \}} dx \\
	&\quad\quad\leq \int_{\{ w_\mu > 0 \} \cap B_r} \sqrt{1 + (\tan\theta_\mu)^2|Dw_\mu|^2} - \cos\theta_\mu 1_{\{ w_\mu > 0\}} dx \\
	&\implies \int_{B_r} \frac{1}{2} |Du_\mu|^2 + (1-\cos\theta_\mu) 1_{\{ u_\mu > 0 \}} dx \\
	&\quad\quad \leq \int_{B_r} \frac{1}{2} (\tan\theta_\mu)^2 |Dw_\mu|^2 + (1-\cos\theta_\mu) 1_{ \{ w_\mu > 0 \}} dx + O(\theta_\mu^3) \\
	&\implies \int_{B_r} |Dv_\mu|^2 + 1_{\{ v_\mu > 0 \}} dx \\
	&\quad\quad \leq \int_{B_r} |D w_\mu|^2 + 1_{\{ w_\mu > 0\}} dx + O(\theta_\mu) \\
	&\quad\quad \leq \int_{B_r} |D w|^2 + 2 D w \cdot D( (1-\eta)(v_\mu - v)) + |D((1-\eta)(v_\mu - v)|^2 dx \\
	&\quad\quad\quad + \int_{B_r} 1_{\{ w > 0 \}} + 1_{\{ \eta < 1 \}} dx + O(\theta_\mu)
	\end{align*}
	In the second implication we used the bound $|Du_\mu| \leq c(n, \gamma) \tan \theta$.  Taking $\mu \to \infty$ gives
	\[
	\int_{B_r} |Dv|^2 + 1_{\{v > 0 \}} dx \leq \int_{B_r} |Dw|^2 + 1_{ \{ w > 0 \}} + 1_{\{ \eta < 1 \}} dx,
	\]
	and then letting $\eta$ approximate $1_{B_r}$ we conclude that
	\[J_{B_{1/2}}(v)\le J_{B_{1/2}}(w).\]
	Finally, for a general $w - v \in W^{1,2}_0(B_{r})$ we can approximate $w - v$ in $W^{1,2}(B_{1/2})$ by a Lipschitz function $\tilde w$ compactly supported in $B_{1/2}$, and apply the above reasoning to $v + \tilde w$.

	\vspace{3mm}

	We now assume the bound \eqref{eqn:blow-up-hypA}, and prove $C^{2,\alpha}_{loc}$ convergence $v_\mu \to v$ near the free-boundary, and convergence as in \eqref{eqn:blow-up-conclA}.  We have
	\begin{equation}\label{eqn:blow-up-19}
	|D^2 v_\mu(z)| = (1+O(\theta_\mu)) \theta_\mu^{-1} |A_{M_\mu}(z + u_\mu(z))|,
	\end{equation}
	and so \eqref{eqn:blow-up-hypA} implies
	\begin{equation}\label{eqn:blow-up-20}
	\sup_{B_1 \cap \{ v_\mu > 0 \} } |D^2 v_\mu| \leq \Lambda < \infty
	\end{equation}
	for some fixed constant $\Lambda$.  Since $v_\mu \to v$ smoothly on compact subsets of $\del \{ v > 0 \}$, we deduce that
	\[
	\sup_{B_1 \cap \{ v > 0 \} } |D^2 v| \leq \Lambda
	\]
	also.  It follows that any tangent solution at any $x \in \del \{ v > 0 \}$ is linear, and hence (by e.g. \cite{AltCaffarelli}) $\del \{v > 0 \}$ is entirely regular.  It also follows that if $x_\mu \to x \in \{ v > 0 \} \cap B_1$, then $\theta_\mu^{-1} |A_{M_\mu}(x_\mu + u_\mu(x_\mu))| \to |Dv(x)|$.
	
	Fix $x \in \del \{ v > 0 \} \cap B_{1/2}$.  We shall show that in some small ball $B_r(x)$ the Hodograph transforms of $v_\mu$ convergence in $C^{2,\alpha}$ to Hodograph transform of $v$.  Given any $\eps > 0$ (to be fixed later to depend only on $n$), we claim there is a radius $r > 0$ so that
	\begin{align}
	r^{-1} |v_\mu(y) - (y - x) \cdot Dv(x)| + |Dv_\mu(y) - Dv(x)| \leq \eps \quad \text{in } B_r(x) \cap \{ v_i > 0 \}, \label{eqn:blow-up-21} \\
	r^{-1} |v(y) - (y - x) \cdot Dv(x)| + r^{-1}|Dv(y) - Dv(x)| \leq \eps \quad \text{in } B_r(x) \cap \{ v > 0 \}. \label{eqn:blow-up-22}
	\end{align}
	
	The second estimate \eqref{eqn:blow-up-22} follows from the regularity of $\del \{ v > 0 \}$, with $r = r(v, x, \eps)$.  The Hausdorff convergence of free boundaries implies there is a sequence $x_\mu \in \del \{ v_\mu > 0 \} \to x$, and from the bound \eqref{eqn:blow-up-20} (and the boundary condition $\del_n v_\mu = -1$ along $\del \{ v_\mu > 0 \}$), we can shrink $r = r(n, \Lambda)$ as necessary to get
	\begin{equation}\label{eqn:blow-up-23}
	r^{-1} |v_\mu(y) - (y - x_i) \cdot Dv_\mu(x_i)| + |Dv_\mu(y) - Dv_\mu(x_i)| \leq \eps \quad \text{in } U_\mu,
	\end{equation}
	where $U_\mu$ is the connected component of $\{ v_\mu > 0 \} \cap B_r(x)$ containing $x_\mu \in \del U_\mu$.  Combining \eqref{eqn:blow-up-22}, \eqref{eqn:blow-up-23}, with the $C^0_{loc}$ convergence $v_i \to v$ we deduce $|Dv_i(x_i) - Dv(x)| \leq c(n)\eps$ and $\{ v_\mu > 0 \} \cap B_r(x) = U_\mu$.  The first estimate \eqref{eqn:blow-up-21} (with $c(n) \eps$ in place of $\eps$) then follows.  This proves our claim.
	
	For convenience, after rotating/translating/dilating, there will be no loss in taking $x = 0$, $r = 1$, and $Dv(x \equiv 0) = e_{n+1}$, so that
	\begin{equation}\label{eqn:blow-up-24}
	|Dv_\mu(y) - e_{n+1}| + |v_\mu(y) - (y_{n+1})_+| \leq \eps \quad \text{in } \{ v_\mu > 0 \} \cap B_1.
	\end{equation}
	Ensuring $\eps(n)$ is sufficiently small, the inverse function theorem implies we can find smooth functions $g_\mu(y_1, \cdots, y_{n}) : \{ y_1 > 0 \} \cap B_1 \to \R$ satisfying
	\begin{align}
	&\{ (v_\mu(y_2, \cdots, y_{n+1}), y_2, \cdots, y_{n+1}) : y \in \{ v_i > 0 \} \cap B_1 \} \cap B_{9/10} \nonumber \\
	&= \{ (y_1, \cdots, y_n, g_\mu(y_1, \cdots, y_n)) : y \in \{ y_1 > 0 \} \cap B_1 \} \cap B_{9/10}, \label{eqn:blow-up-26}
	\end{align}
	and similarly we can find a $g$ satisfying \eqref{eqn:blow-up-26} with $g$, $v$ in place of $g_\mu$, $v_\mu$.  From \eqref{eqn:blow-up-24}, we have
	\begin{equation}\label{eqn:blow-up-25}
	|g_\mu(y_1, \cdots, y_n) - y_1| + |Dg_\mu - e_1| \leq c(n)\eps \quad \text{in } \{ y_1 > 0 \} \cap B_{9/10}
	\end{equation}
	and similarly for $g$.  The $C^0_{loc}$ convergence $v_\mu \to v$ implies that $g_\mu \to g$ in $C^0(B_{9/10}^+)$.
	
	By differentiating \eqref{eqn:blow-up-26}, we have for every $2 \leq i, j \leq n$ the relations
	\begin{align}
	&1 = D_{n+1} v D_1 g, \quad 0 = D_i v + D_{n+1} v D_i g \\
	&0 = D_{n+1, n+1}^2 v D_1 g D_1 g + D_{n+1} v D^2_{11} g \\
	&0 = D^2_{n+1, i} v D_1 g + D^2_{n+1, n+1} v D_1 g D_i g + D_{n+1} v D^2_{1, i} g \\
	&0 = D^2_{ij} v + D^2_{i, n+1} v D_j g + D^2_{j, n+1} D_i g + D^2_{n+1, n+1} v D_i g D_j g + D_{n+1} v D^2_{ij} g,
	\end{align}
	and the same with $g_\mu, v_\mu$ in place of $g, v$.  From the above, the bound \eqref{eqn:blow-up-20}, and the estimates \eqref{eqn:blow-up-24}, we get the uniform bound
	\begin{equation}\label{eqn:blow-up-27}
	\sup_{B_{9/10}^+} |D^2 g_\mu| \leq c(n) \Lambda,
	\end{equation}
	and since $D v(0) = e_{n+1}$, we have
	\begin{equation}\label{eqn:blow-up-28}
	Dg(0) = e_1, \quad |D^2 g(0)| = |D^2 v(0)|.
	\end{equation}
	
	We show that the $g_\mu$ satisfy a good elliptic PDE with Neumann-type boundary conditions, which we use to establish a priori $C^{2,\alpha}$ bounds.  For a general smooth function $g(y_1, \cdots, y_n) : B_1^+ \to \R$, define
	\[
	\graph^\theta(g) = \{ (\tan\theta y_1, y_2, \cdots, y_n, g(y_1, \ldots, y_n)) : y_1 > 0 \},
	\]
	so that we have $M_\mu \cap B_{8/10} = \graph^{\theta_\mu}(g_\mu) \cap B_{8/10}$.  Denote by
	\[
	D^\theta g = \left( \frac{ D_1 g}{\tan \theta}, D_2 g, \cdots, D_n g \right).
	\]
	Then the upward unit normal $\nu$ and volume form $d Vol$ for $\graph^\theta(g)$ have the expressions
	\begin{align}
	&\nu(g) = \frac{1}{(1+|D^\theta g|^2)^{1/2}}\left(-\frac{D_1g}{\tan\theta}, -D_2g ,\cdots, -D_n g,1\right),\\
	&d\Vol = (1+|D^\theta g|)^{1/2} dx.
	\end{align}
	Note that by \eqref{eqn:blow-up-25} we have
	\begin{equation}
	|\tan \theta_\mu (1+|D^{\theta_\mu} g_\mu|^2)^{1/2} - 1|_{C^0(B_{9/10}^+)} \leq c(n) \eps, \quad \text{for } \mu \gg 1.
	\end{equation}
	
	Thus, in $B_{8/10}^+$ each $g_\mu$ is stationary for the area functional
	\[
	A(\graph^{\theta_\mu}(g_\mu)) = \int (1 + |D^{\theta_\mu} g|)^{1/2}
	\]
	with boundary condition $\nu(g_\mu) \cdot (-e_1) = \cos\theta_\mu$ when $y_1 = 0$.  A standard computation shows that $g_\mu$ is a solution to the equation
	\begin{equation}\label{eq.for.g}
	\begin{cases}
	\mathrm{div}\left(\frac{1}{\tan\theta_\mu(1+|D^{\theta_\mu} g_\mu|^2)^{1/2}}\left(\frac{D_1 g_\mu}{\tan^2\theta_\mu},D_2 g_\mu,\cdots,D_n g_\mu\right)\right)=0 \quad &\text { in }\{y_1 \geq 0\},\\
	(D_1 g_\mu )^2= (|Dg_\mu|^2+1) / 2\quad &\text { on }\{y_1=0\}.
	\end{cases}
	\end{equation}
	
	Note that \eqref{eq.for.g} is a divergence form elliptic equation, but whose coefficients degenerate as $\theta_\mu\to 0$. Denote these coefficients by $(a_{ij}(Dg_\mu))_{1\le i,j \le n}$, $a_{ij}: \RR^n \to \RR$, so the equation for $g_\mu$ becomes $D_i(a_{ij}(Dg_\mu)D_j g_\mu)=0$, for
	\begin{equation}\label{eq.aij}
	\begin{split}
	&a_{11}(p)= \frac{1}{\tan^3\theta_\mu (1+\frac{p_1^2}{\tan^2\theta_\mu}+p_2^2+\cdots+p_n^2)^{1/2}},\\
	&a_{ii}(p) = \frac{1}{\tan\theta_\mu (1+\frac{p_1^2}{\tan^2\theta_\mu}+p_2^2+\cdots+p_n^2)^{1/2}}, \quad 2\le i\le n,
	\end{split}
	\end{equation}
	and $a_{ij}=0$ for all other indices.
	
	Fix $1\le k\le n$ an integer, and set $w_\mu =D_k g_\mu$. Differentiating \eqref{eq.for.g} in $y_k$, we obtain that $w_\mu$ is a solution to the equation
	\begin{equation}\label{eq.for.w}
	D_i(\bar a_{ij} D_j w_\mu) = 0 \text { in }y_1\geq 0
	\end{equation}
	and when $k\ge 2$, we have the boundary condition
	\begin{equation}\label{eq.boundary.w}
	\mathbf{b}\cdot Dw_\mu = 0 \text { on }\{y_1=0\}.
	\end{equation}
	Here
	\[\bar a_{ij } = a_{ij}(Dg_\mu)+\sum_p (D_j a_{ip})(Dg_\mu) D_p g_\mu , \quad \mathbf {b}=D_1 g_\mu e_1 - \sum_{j=2}^n D_j g_\mu e_j.\]
	The vector $\mathbf{b}$ is uniformly close to $c(n)\eps$ as $\theta_\mu \to 0$, and we can compute the coefficients $\bar a_{ij}$ as
	\begin{align}
	&\bar a_{11}=  \frac{1+ (D_2 g_\mu)^2 + \cdots + (D_n g_\mu)^2}{\tan^3\theta_\mu \left(1+|D^{\theta_\mu} g_\mu|^2\right)^{3/2}} ,\label{eq.tilde.a11}\\
	&\bar a_{1j}=\bar a_{j1} = -\frac{D_1 g_\mu D_j g_\mu}{\tan^3\theta_\mu(1+|D^{\theta_\mu} g_\mu|^2)^{3/2}}, \quad j\ge 2\label{eq.tilde.a1j}\\
	&\bar a_{jj} = \frac{1}{\tan\theta_\mu (1+|D^{\theta_\mu} g_\mu|^2)^{1/2}} - \frac{(D_j g_\mu)^2}{\tan\theta_\mu (1+|D^{\theta_\mu} g_\mu|^2)^{3/2}} , \quad j\ge 2\\
	&\bar a_{ij} = -\frac{D_i g_\mu D_j g_\mu}{\tan\theta_\mu (1+|D^{\theta_\mu} g_\mu|^2)^{3/2}}, \quad i,j\ge 2, i\ne j.
	\end{align}
	Using \eqref{eqn:blow-up-25}, \eqref{eqn:blow-up-27}, we conclude that
	\begin{equation}\label{eqn:blow-up-30}
	|\bar a_{ij} - \delta_{ij}|_{C^0(B_{9/10}^+)} \leq c(n)\eps, \quad \text{and} \quad |\bar a_{ij}|_{C^\alpha(B_{9/10}^+)} \leq c(n)\Lambda,
	\end{equation}
	provided $\mu \gg 1$.  Therefore (ensuring $\eps(n)$ is small) \eqref{eq.for.w} is a divergence form uniformly elliptic equation with uniformly Holder coefficients.
	
	A key observation is that \eqref{eq.boundary.w} is the natural boundary condition for \eqref{eq.for.w}, in the sense that for every $\varphi\in C_0^1(B_{9/10} \cap \RR^n=B_{9/10} \cap \{y_{n+1}=0\})$, we have
	\begin{equation}\label{eq.weak.form.for.w}
	\int_{\{y_1\ge 0, y_{n+1}=0\}} \bar a_{ij} D_j w_\mu D_i \varphi = 0.
	\end{equation}
	Indeed, we have that
	\begin{align*}
	0&=-\int_{\{y_1\ge 0,y_{n+1}=0\}} \varphi D_i (\bar a_{ij}D_j w_\mu) \\
	&= \int_{\{y_1\ge 0, y_{n+1}=0 \}} \bar a_{ij}D_j w D_i \varphi  - \int_{\{y_1=0,y_{n+1=0}\}} \varphi \sum_{j=1}^n \bar a_{1j}D_j w_\mu,
	\end{align*}
	and \eqref{eq.for.g} and \eqref{eq.tilde.a11} implies that
	\[\bar a_{11} = \frac{1}{\tan^3\theta (1+|D^\theta g|^2)^{\frac 32}} (D_1 g)^2.\]
	Combined with \eqref{eq.tilde.a1j}, \eqref{eq.boundary.w} implies that $\sum_{j=1}^n \bar a_{1j}D_j w = 0$ on $\{x_1=x_{n+1}=0\}$.
	
	Thus, by \eqref{eqn:blow-up-30}, \eqref{eq.weak.form.for.w} we can apply $C^{1,\alpha}$ Schauder theory to get the bound
	\begin{equation}\label{eqn:blow-up-31}
	[D^2_{ij} g_\mu]_{\alpha, B_{8/10}^+} \leq c(n) \Lambda, \quad (i, j) \neq (1, 1).
	\end{equation}
	On the other hand, by rewriting \eqref{eq.for.w} in non-divergence form, we get that $g_\mu$ satisfies the equation $\bar a_{ij} D^2_{ij} g_\mu = 0$ on $\{ y_1 \geq 0 \}$, and so we can write
	\begin{equation}\label{eqn:blow-up-32}
	D_{11}^2 g_\mu = -\bar a_{11}^{-1} \sum_{i + j > 2} \bar a_{ij} D^2_{ij} g_\mu,
	\end{equation}
	and then combine \eqref{eqn:blow-up-27}, \eqref{eqn:blow-up-31}, \eqref{eqn:blow-up-32} to deduce
	\begin{equation}\label{eqn:blow-up-33}
	|g_\mu|_{C^{2,\alpha}(B_{8/10}^+)} \leq c(n)\Lambda.
	\end{equation}
	
	By Azela-Ascoli and the $C^0$ convergence $g_\mu \to g$, we deduce that $g_\mu \to g$ in $C^{2,\alpha}(B_{8/10}^+)$.  In particular, if $x_\mu \in \overline{\{ v_\mu > 0 \}} \to 0 \in \del \{ v > 0 \}$, then from \eqref{eqn:blow-up-19}, \eqref{eqn:blow-up-33}, \eqref{eqn:blow-up-28} we get
	\begin{align*}
	\theta_\mu^{-1} |A_{M_\mu}(x_\mu + u_\mu(x_\mu))| &= (1+o(1)) |D^2 v_\mu(x_\mu)| \\
	&= |D^2 g_\mu(0)| + o(1) \\
	&\to |D^2 g(0)| = |D^2 v(0)|.
	\end{align*}
	This proves \eqref{eqn:blow-up-conclA}.
\end{proof}

\subsection{A priori estimates} \label{sec:a priori}

For smooth minimizers, we can ``bootstrap'' the convergence of Proposition \ref{prop:blow-up} to show that the non-degeneracy estimates \eqref{eqn:graph-prop-hyp} hold for a uniform choice of $\gamma$.

\begin{lemm}[Bootstrap non-degeneracy near free-boundary]\label{lem:bootstrap}
	There are constants $\theta_0(n)$ and $d_0(n)$ so that the following holds.  Let $\Omega$ be a smooth minimizer of $\cA^\theta$ in $B_1 \subset \R^{n+1}$ with $\theta \in (0, \theta_0)$.  Then writing $M = \del\Omega \cap \R^{n+1}_+$ we have
	\begin{equation}\label{eqn:bootstrap-concl1}
	\frac{1}{2c_0} \tan\theta d(\pi(x), \del M) < x_1 < 2 \tan\theta d(\pi(x), \del M)
	\end{equation}
	for all $x \in M \cap B_{1/2} \cap B_{d_0}(\del M)$, where $c_0$ is as in Lemma \ref{lemm.C0.bound.on.AC.minimizer}.
\end{lemm}

\begin{proof}
	Given a smooth minimizer $\Omega$ of $\cA^\theta$ in some ball $B_R$, and taking $M = \del \Omega \cap \R^{n+1}_+$ and $r < R$, define $D(M, B_r)$ to be the largest number $D$ so that \eqref{eqn:bootstrap-concl1} holds with $B_{D}(\del M)$ in place of $B_{d_0}(\del M)$ and $B_r$ in place of $B_{1/2}$.  Note that $D(M, B_r)$ is always positive (by our assumptions that $M$ is smooth and $r < R$), and that $r \mapsto D(M, B_r)$ is decreasing.  Note also the scaling $D(\lambda M, B_{\lambda r}) = \lambda D(M, B_r)$.
	
	After replacing $B_1$ with $B_{3/4}$, there will be no loss of generality in assuming $D(M, B_1) > 0$.  We shall prove: there are $C(n)$, $\theta_0(n)$ positive so that provided $\theta \leq \theta_0$, then we have the bound
	\begin{equation}\label{eqn:bootstrap-1}
	(1-r) D(M, B_r)^{-1} \leq C \quad \forall r \in (0, 1),
	\end{equation}
	which clearly will imply the Lemma.
	
	Suppose, towards a contradiction, there is no $C(n)$, $\theta_0(n)$ which makes \eqref{eqn:bootstrap-1} true.  Then there are sequences $\theta_\mu \to 0$, and minimizers $\Omega_i$ of $\cA^{\theta_i}$ for which (writing $M_i = \del\Omega_i \cap \R^n$)
	\[
	\sup_{r \in (0, 1)} \,\, (1 - r) D(M_i, B_r)^{-1} \to \infty .
	\]
	Choose $r_i \in (0, 1)$ so that
	\[
	(1 - r_i) D(M_i, B_{r_i})^{-1} \geq \frac{1}{2} \sup_{r \in (0, 1)} \,\, (1 - r) D(M_i, B_r)^{-1}.
	\]
	Choose $z_i \in \overline{B_{r_i}}$ with $d(z_i, \del M_i) = D(M_i, B_{r_i})$ so that
	\begin{align*}
	\text{either}\quad  \frac{1}{2c_0} \tan\theta_i d(\pi(z_i), \del M_i) = z_{i,1}, \quad \text{or}\quad z_{i, 1} = 2 \tan\theta_i d(\pi(z_i), \del M_i)
	\end{align*}
	($z_{i,1}$ being the first component of $z_i$).  Set $d_i = d(\pi(z_i), \del M) \leq D(M_i, B_{r_i})$.  Clearly we have $d_i \to 0$ and
	\[
	d(\pi(z_i), \del B_1) d_i^{-1} \geq (1 - r_i) d_i^{-1} \geq (1-r_i)D(M_i, B_{r_i})^{-1} \to \infty.
	\]
	
	Define the dilated/translated domains $\Omega_i' = (\Omega_i - \pi(z_i)) / d_i$, surfaces $M_i' = (M_i - \pi(z_i))/d_i$, and points $z_i' = (z_i - \pi(z_i))/r_i$.  Given any fixed $R > 0$, then for $i \gg 1$ the $\Omega_i'$ will be minimizers of $\cA^{\theta_i}$ in $B_R(0) \equiv B_R(\pi(z_i'))$ with
	\begin{equation} \label{eqn:bootstrap-2}
	D(M_i', B_R) \geq 1/4,
	\end{equation}
	but for which
	\begin{equation}\label{eqn:bootstrap-3}
	\text{either} \quad \frac{1}{2c_0} \tan\theta_i d(0, \del M_i) = z_{i, 1}' \quad \text{or} \quad z_{i, 1}' = 2 \tan\theta_i d(0, \del M_i').
	\end{equation}
	
	From \eqref{eqn:bootstrap-2}, Lemma \ref{lem:graph-prop}, Proposition \ref{prop:blow-up}, for any $R > 0$ there are $i_0(R)$ and $\eps(R)$ and Lipschitz functions $u_i : B_R^n \to \R$ so that $M_i' \cap B_R = \graph_{\R^n}(u_i)$ in $\{ 0 < x_1 < \eps(R)\}$ whenever $i > i_0(R)$.  Taking a diagonal subsequence, the functions $u_i/\tan\theta_i$ converge in $(C^\alpha_{loc} \cap W^{1,2}_{loc})(\R^n)$ to an entire Lipschitz minimizer $v : \R^n \to \R$ of the Alt-Caffarelli functional $J$.  Lemma \ref{lemm.C0.bound.on.AC.minimizer} and the local Hausdorff convergence $\del \{ u_i > 0 \} \to \del \{ v > 0 \}$ imply that
	\[
	\frac{1}{2c_0} \tan\theta_i d(z, \del \{ u_i > 0 \}) < u_i(z) < 2 \tan\theta_i d(z, \del \{ u_i > 0 \})
	\]
	for all $z \in B_1^n$ and all $i \gg 1$, which contradicts \eqref{eqn:bootstrap-3} since $z_{i, 1}' = u_i(0)$ and $\del \{ u_i > 0 \} = \del M_i'$.
\end{proof}

Most importantly for us, in low dimensions (whereever $1$-homogenous minimizers of the Alt-Caffarelli energy are linear) we can obtain a priori curvature estimates near the boundary also.
\begin{lemm}[Curvature estimates]\label{lem:curvature-est}
	There are absolute constants $C$, $\eps$, $\theta_0$ so that the following holds.  Take $n \leq 4$, and let $\Omega$ be a smooth minimizer of $\cA^\theta$ in $B_1 \subset \R^{n+1}$, for $\theta \in (0, \theta_0)$, and write $M = \del \Omega \cap \R^{n+1}_+$.  Then, assuming $0 \in \del M$, we have the curvature bound $\sup_{M \cap B_\eps(\R^n) \cap B_{1/8}}  \theta^{-1}|A_M| \leq C$.
\end{lemm}

\begin{proof}
	After replacing $B_1$ with $B_{1/4}$ there will be no loss in assuming that $\sup_{B_1} |A_M| < \infty$.  Moreover, ensuring $\theta_0$ is sufficiently small, by Lemmas \ref{lem:bootstrap}, \ref{lem:graph-prop} there is also no loss in assuming we have a Lipschitz function $u : B_1^n \to \R$ and absolute constant $\eps > 0$ so that $M = \graph_{\R^n}(u)$ in $B_1 \cap \{ 0 < x_1 < \eps\}$ and $\Lip(u) \leq c \theta$ and
	\begin{equation}\label{eqn:curvature-est-1}
	\frac{1}{c} \theta d(z, \del \{ u > 0 \}) \leq u(z) \leq 2 c \theta d(z, \del \{ u > 0 \})
	\end{equation}
	for all $z \in B_1^n \cap \{ u > 0 \}$.  Here $c$ is an absolute constant.
	
	We will show that $\theta^{-1} |A_M(x)| (1-|x|)$ is bounded by an absolute constant, provided $\theta$ is sufficiently small, which will clearly prove the Lemma.  Suppose otherwise: there are sequences $\theta_i \to 0$, $\Omega_i$ minimizing $\cA^{\theta_i}$ in $B_1$, so that if $M_i = \del\Omega_i \cap \R^{n+1}_+$ then 
	$$\sup_{x \in B_1} \theta_i^{-1} |A_{M_i}(x)| (1 - |x|) \to \infty$$
	as $i \to \infty$.  Choose $x_i \in B_1$ satisfying 
	\[
	\theta_i^{-1} |A_{M_i}(x_i)| (1 - |x_i|) \geq \frac{1}{2} \sup_{x \in B_1} \theta_i^{-1} |A_{M_i}(x)| (1 - |x|).
	\]
	Let $\lambda_i = \theta_i^{-1} |A_{M_i}(x_i)|$.  Write $u_i$ for the graphing functions for $M_i$ as in the first paragraph. We break into two cases. 
	
	\textbf{Case 1:} $\sup_i \lambda_i^{-1} d(x_i, \del M_i) < \infty$.  Define the rescaled domains $\Omega_i' = \lambda_i(\Omega_i - \pi(x_i))$, surfaces $M_i' = \lambda_i (M_i - \pi(x_i)) \equiv \del \Omega_i' \cap \del \R^{n+1}_+$, and points $x_i' = \lambda_i(x_i - \pi(x_i))$.  Then for a suitable $R_i \to \infty$, the $\Omega_i'$ are minimizers of $\cA^{\theta_i}$ in $B_{R_i}(0) \equiv B_{R_i}(\pi(x_i'))$, satisfying
	\begin{equation}\label{eqn:curvature-est-2}
	\sup_i d(0, \del M_i') < \infty, \quad \theta_i^{-1} |A_{M_i'}(x_i')| = 1, \quad \sup_{B_{R_i}} \theta_i^{-1} |A_{M_i'}| \leq 2.
	\end{equation}
	Moreover, if we let $u_i'(z) = \lambda_i u_i( (z - \pi(x_i))/\lambda_i)$, then $M_i' \cap B_{R_i} = \graph_{\R^n}(u_i')$, $\Lip(u_i') \leq c(\rho) \theta_i$, and the $u_i'$ continue to satisfy \eqref{eqn:curvature-est-1} on $B_{R_i}^n \cap \{ u_i' > 0 \}$.
	
	We can therefore apply Proposition \ref{prop:blow-up} (after passing to a subsequence as necessary) to deduce there is entire Lipschitz minimizer $v : \R^n \to \R$ of the Alt-Caffarelli functional $J$ so that $\theta_i^{-1} u_i' \to v$ in $(C^\alpha_{loc} \cap W^{1, 2}_{loc})(\R^n)$.  By Lemma \ref{theo.jerison.savin} and our restriction $n \leq 4$, we must have $v(z) = ((z - z_0) \cdot n)_+$ for some $z_0 \in \R^n$ and unit vector $n$, and in particular we have $D^2 v \equiv 0$.  On the other hand, \eqref{eqn:curvature-est-2} and Proposition \ref{prop:blow-up} imply we have
	\[
	\sup_{B_1} \theta_i^{-1} |A_{M_i'}| \to \sup_{B_1^n \cap \{ v > 0 \}} |D^2 v| \equiv 0,
	\]
	which contradicts our choice of $\theta_i^{-1} |A_{M_i'}(0)| = 1$.

	\textbf{Case 2:} $\sup_i \lambda_i^{-1} d(x_i, \del M_i) = \infty$.  After passing to a subsequence we can assume $$\lim_i \lambda_i^{-1} d(x_i, \del M_i) = \infty\,.$$  Define the functions
	\[
	u_i'(z) = \lambda_i (u_i((z - \pi(x_i))/\lambda_i) - x_{i, 1}), 
	\]
	so that $M_i' := \lambda_i (M_i - x_i)$ are the graphs of $u_i'$.  Then for a sequence $R_i \to \infty$, the $u_i'$ are smooth solutions to the minimal surface equation in $B_{R_i}^n$ satisfying
	\[
	\Lip(u_i') \leq c \theta_i, \quad u_i'(0) = 0, \quad \theta_i^{-1} |D^2 u_i'(0)| = 1 + O(\theta_i).
	\]
	
	By standard interior estimates and the structure of the minimal surface equation, after passing to a further subsequence, we have $C^2_{loc}(\R^n)$ convergence of $\theta_i^{-1} u_i' \to v$ for some harmonic function $v : \R^n \to \R$ satisfying
	\[
	\Lip(v) \leq c, \quad v(0) = 0, \quad |D^2 v(0)| = 1.
	\]
	But now the Liouville theorem for harmonic functions together with the first condition implies $v$ is linear, which is a contradiction with the third condition.
\end{proof}

\subsection{Proof of Theorems \ref{thm:graphical.convergence.to.AC.minimizer},  \ref{thm:AC.bernstein.implies.capillary.bernstein}, \ref{thm:general-blow-up}}

Theorems \ref{thm:graphical.convergence.to.AC.minimizer}, \ref{thm:general-blow-up} follow by directly combining Lemma \ref{lem:graph-prop-cones}, Proposition \ref{prop:blow-up}, with Lemmas \ref{lem:bootstrap} and \ref{lem:curvature-est}.  The basic strategy for Theorem \ref{thm:AC.bernstein.implies.capillary.bernstein} involves three key ingredients: we can write $M_i = \graph(u_i)$ and blow-up $\theta_i^{-1} u_i \to v$ for some $1$-homogenous $v$ minimizing the Alt-Caffarelli energy (Theorem \ref{thm:graphical.convergence.to.AC.minimizer}); in low dimensions $v$ must be linear, and so the improved convergence of Proposition \ref{prop:blow-up} implies $\theta_i^{-1} |A_{M_i \cap \del B_1}| \to 0$; a rigidity-type theorem for minimal capillary surfaces $\Sigma$ in the sphere with small angle $\theta$ and $\theta^{-1}|A_\Sigma|$ small. We first state the rigidity theorem.
\begin{lemm}[Rigidity of almost-planar capillary cones with small angle]\label{lem:rigidity}
	There is a $\delta(n)$ so that the following holds.  Let $M$ be a stationary capillary cone in $\R^{n+1}_+$ making contact angle $\theta > 0$ with $\R^n$, and with smooth link $\Sigma$.  Suppose the outward conormal $\eta_\Sigma$ and second fundamental form $A_\Sigma$ of $\Sigma$ satisfy
	\[
	|\eta_\Sigma + e_{n+1}| \leq \delta, \quad \theta^{-1} |A_\Sigma| \leq \delta.
	\]
	Then $M$ is a half-plane.
\end{lemm}

(Note, though we don't explicitly state it as a hypothesis, $\eta_\Sigma$ can only be close to $-e_{n+1} \in \R^n \equiv \del \R^{n+1}_+$ provided $\theta$ is very small.)

\begin{proof}
	We first note the following trace-type inequality: provided $\delta$ is small, we have for any non-negative $f \in C^1(\Sigma)$:
	\begin{align}
	\int_{\del \Sigma} f &\leq 2 \int_{\del\Sigma} f (- \eta \cdot e_{n+1}) = 2 \int_\Sigma \Div_\Sigma(-f e_{n+1}) \nonumber \\
	&= 2\int_\Sigma \nabla_\Sigma f \cdot (-e_{n+1}) + f \Div_\Sigma(-e_{n+1}) 
	\leq 2 \int_\Sigma |\nabla_\Sigma f| + (n-1) f, \label{eqn:rigidity-1}
	\end{align}
	having used that $\Div_\Sigma(e_{n+1}) = n-1$.
	
	We now take Simons' equation on $\Sigma$ (see \cite[Theorem 5.3.1]{Simons}),
	\[
	\Delta_\Sigma A_\Sigma = (n-1 - |A_\Sigma|^2)A_\Sigma ,
	\]
	and inner product with $A_\Sigma$ and integrate over $\Sigma$, to obtain
	\begin{equation}\label{eqn:rigidity-2}
	\int_\Sigma |\nabla_\Sigma A_\Sigma|^2 + (n-1-|A_\Sigma|^2)|A_\Sigma|^2 = \int_{\del\Sigma} \frac{1}{2} \frac{\del |A_\Sigma|^2}{\del \eta}.
	\end{equation}
	
	By \cite[Lemma C.2]{LiZhouZhu}, we have
	\begin{equation}\label{eqn:rigidity-3}
	\left| \frac{\del |A_\Sigma|^2}{\del \eta} \right| \leq c(n) \cot\theta |A_\Sigma|^3.
	\end{equation}
	Therefore, combining \eqref{eqn:rigidity-3} with \eqref{eqn:rigidity-2}, and using \eqref{eqn:rigidity-1} with our assumption $\theta^{-1} |A_\Sigma| \leq \delta$, we get
	\begin{align*}
	\int_\Sigma |\nabla_\Sigma A_\Sigma|^2 + (n-1)|A_\Sigma|^2 
	&\leq c(n) \delta \int_{\del\Sigma} |A_\Sigma|^2 + \int_\Sigma |A_\Sigma|^4\\
	&\leq c(n) \delta \int_\Sigma 2 |A_\Sigma| |\nabla_\Sigma |A_\Sigma|| + (n-1) |A_\Sigma|^2 + \delta \theta \int_\Sigma |A_\Sigma|^2 \\
	&\leq c(n) \delta \int_\Sigma |\nabla_\Sigma A_\Sigma|^2 + |A_\Sigma|^2 ,
	\end{align*}
	having used Kato's inequality $|\nabla_\Sigma |A_\Sigma|| \leq |\nabla_\Sigma A_\Sigma|$ in the last line.  Ensuring $\delta(n)$ is small, we deduce $A_\Sigma \equiv 0$, as desired.
\end{proof}

\begin{proof}[Proof of Theorem \ref{thm:AC.bernstein.implies.capillary.bernstein}]
	Suppose, towards a contradiction, there are sequences $\theta_i \to 0$ and cones $\Omega_i$ minimizing $\cA^{\theta_i}$, so that each $M_i := \del\Omega_i \cap \R^{n+1}_+$ is smooth away from $0$ but non-planar.  For $i \gg 1$, we can apply Lemma \ref{lem:bootstrap} at points in $\del M_i \cap \del B_1$ to get the non-degeneracy condition \eqref{eqn:graph-prop-cones-hyp} with some absolute constant $\gamma$.  We can then use Lemma \ref{lem:graph-prop-cones} and Proposition \ref{prop:blow-up} to deduce $M_i = \graph_{\R^n}(u_i)$, where $u_i/\theta_i \to v$ for some $1$-homogenous minimizer $v$.
	
	By our dimensionality hypothesis $v$ must be linear, and without loss of generality $v(x) = (e_{n+1} \cdot x)_+$.  By Lemma \ref{lem:curvature-est} and Proposition \ref{prop:blow-up}, we have $\sup_{\del B_1} \theta_i^{-1} |A_{M_i}| \to \sup_{\del B_1} |D^2 v| = 0$, and that $\theta_i^{-1} u_i \to (e_{n+1} \cdot x)_+$ in $C^{2,\alpha}_{loc}(B_2 \setminus B_{1/2})$ at the level of Hodograph transform, which implies that the conormals $\eta_{M_i} \to -e_{n+1}$ uniformly in $\del B_1$.  Lemma \ref{lem:rigidity} then tells us that the $M_i$ are planar for $i \gg 1$, which is a contradiction.
\end{proof}

\section{The case $4\le n \le 6$ with $\theta$ close to $\pi/2$}

In this section we prove case (3) of Theorem \ref{thm:main.cone}. In fact, we prove the following slightly stronger result.

\begin{theo}\label{thm:theta.close.pi/2}
	Suppose $4\le n\le 6$. There exists $\theta_1=\theta_1(n)$ such that if $\Omega\subset \RR^{n+1}$ is a stable minimal cone for $\cA^\theta$, $M= \partial \Omega\cap \RR^{n+1}_+$ is smooth away from $0$, and $\theta\in (\frac{\pi}{2}-\theta_1, \frac{\pi}{2}]$, then $M$ is flat.
\end{theo}

\begin{rema}
	It will be clear from the proof that $\theta_1$ can be explicitly computed depending on $n$ (compare to \cite[Appendix C]{LiZhouZhu}).
\end{rema}

We first recall the following consequence of the Simons's equation on a minimal cone. 

\begin{lemm}\label{lemm.Simons.on.min.cone}
	Let $M^n\subset \RR^{n+1}$ be a minimal cone. Then for every $\lambda\in (0,1)$, we have that
	\begin{equation}\label{eq.modified.simons}
	\Delta(\tfrac 12 |A|^2) + |A|^4 \ge 2\lambda |x|^{-2} |A|^2 + \left[(1-\lambda)(1+\frac 2n)+\lambda\right] |\nabla |A||^2
	\end{equation}
	on $M$.
\end{lemm}

\begin{proof}
	Fix a point $p\in M$ and take normal coordinates $\{x_i\}_{i=1}^n$ around $p$, then the Simons equation gives
	\[\Delta(\tfrac12 |A|^2) + |A|^4 = \sum_{i,j,k=1}^n A_{ij,k}^2.\]
	We write the right hand side of the above as:
	\begin{multline*}
	\Delta(\tfrac 12 |A|^2) +|A|^4 = \lambda \left(\sum_{i,j,k=1}^n A_{ij,k}^2 - \sum_{i,j,k=1}^n |A|^{-2} (A_{ij}A_{ij,k})^2\right) \\
	+ (1-\lambda) \sum A_{ij,k}^2 + \lambda \sum |A|^{-2} (A_{ij}A_{ij,k})^2.
	\end{multline*}
	Note that $\sum |A|^{-2} (A_{ij}A_{ij,k})^2 = |\nabla |A||^2$. On one hand, since $M$ is a cone, we have that (see \cite[B.9]{LeonSimon})
	\[\sum_{i,j,k=1}^n A_{ij,k}^2 - \sum_{i,j,k=1}^n |A|^{-2} (A_{ij}A_{ij,k})^2 \ge 2 |x|^{-2} |A(x)|^2.\]
	On the other hand, we have the following Kato type inequality (see \cite[(2.22)]{ColdingMinicozzi}):
	\[\sum A_{ij,k}^2\ge (1+\frac 2n) |\nabla |A||^2.\]
	\eqref{eq.modified.simons} follows from combining the above two estimates.
\end{proof}

Next, we establish the following trace type inequality on $M$. This was proven in \cite[Lemma C.3]{LiZhouZhu}. We include it here for completeness. (Compare to Lemma \ref{lem:rigidity}, where a similar trace type inequality holds when $\theta$ is close to $0$).

\begin{lemm}\label{lemma.trace.ineq}
	Suppose $M^n\subset \RR^{n+1}_+$ is a smoothly immersed hypersurface, meeting $\partial \RR^{n+1}_+$ at constant angle $\theta$. Then for any $u\in W^{1,1}(M)$, we have
	\begin{equation}\label{eq.trace}
	\int_{\partial M} u \le \frac{1}{\sin \theta}\int_M |\nabla u| + |H_M u|.
	\end{equation}
\end{lemm}

\begin{proof}
	For $R>0$, let $\varphi_R: [0,\infty)\to \RR$ be a smooth function such that $\varphi_R(t)=1$ when $t\le R$, $\varphi_R(t)=0$ when $t\ge 2R$, and $|\varphi_R'|\le \tfrac 2R$. Consider the vector field $\xi = -\varphi_R(x_1)\partial_1$. By assumption, we have that $\eta\cdot \xi = \sin\theta$ along $\partial M$. Therefore,
	\begin{align*}
	\int_{\partial M} u  &= \frac{1}{\sin\theta} \int_{\partial M} u\eta\cdot\xi\\
	&= \frac{1}{\sin\theta} \int_M\Div_M (u \xi^T)\le \frac{1}{\sin\theta} \int_M |\nabla u||\xi| + u\Div_M \xi + |u H_M| \\
	&\le \frac{1}{\sin \theta} \int_M |\nabla u| + \frac 2R u + |H_M u|.
	\end{align*}
	Sending $R\to \infty$ gives the result.
\end{proof}

\begin{proof}[Proof of Theorem \ref{thm:theta.close.pi/2}]
	For $p\in (\tfrac 12,1)$ to be chosen later, by \eqref{eq.modified.simons}, we have that (here and after we use $r=|x|$)
	
	\begin{align*}
	\frac12 \Delta (|A|^{2p}) &= p |A|^{2p-1} \Delta |A| + p(2p-1) |A|^{2p-2} |\nabla |A||^2 \\
	&\ge p |A|^{2p-1} \left[- |A|^3 + 2\lambda r^{-2} |A| + |A|^{-1}((1-\lambda)\tfrac 2n)|\nabla |A||^2\right]\\
	&\quad +p(2p-1)|A|^{2p-2}|\nabla |A||^2\\
	&=p(2p-2+ \lambda + (1-\lambda)(1+\tfrac 2n)) |A|^{2p-2} |\nabla |A||^2 - p |A|^{2p+2} + 2p\lambda |A|^{2p}r^{-2} 
	\end{align*}
	
	For a compactly supported radial Lipschitz function $f=f(r)$, plug $f|A|^p$ into \eqref{eq:stable.euclidean} and obtain:
	\begin{equation}\label{eq.f|A|^q.in.simons}
	\begin{split}
	\cot\theta &\int_{\partial M} A(\eta,\eta) |A|^{2p}f^2 + \int_M |A|^{2p+2} f^2 \le \int_M |\nabla (f|A|^p)|^2 \\
	& = \int_M p^2 |A|^{2p-2} |\nabla |A||^2 f^2 + |A|^{2p} |\nabla f|^2 + \frac12 \langle \nabla |A|^{2p}, \nabla (f^2)\rangle\\
	&= \int_M p^2 |A|^{2p-2} |\nabla |A||^2 f^2 + |A|^{2p} |\nabla f|^2 -\int_M f^2 \frac 12 \Delta (|A|^{2p}) + \int_{\partial M} f^2 \frac{\partial |A|^{2p}}{\partial \eta}.
	\end{split}
	\end{equation}
	
	We note that each integral on the right hand side of the above inequality is finite. Here one only needs to check the integral involving $|A^{2p-2}||\nabla |A||^2$ is finite. To see this, recall that the Simons equation implies that
	\[c(n)|\nabla |A||^2 \le |A|\Delta |A| + |A|^4,\]
	for some $c(n)>0$. Thus, we have that
	\[c(n)|A|^{2p-2} |\nabla |A||^2 \le |A|^{2p-1}\Delta |A| + |A|^{2p+3}.\]
	Now both terms on the right hand side are integrable on $M\cap S^n(1)$, and hence $|A|^{2p-2} |\nabla |A||^2 f^2$ is integrable on $M$ if $f$ is compactly supported.
	
	We treat boundary term in \eqref{eq.f|A|^q.in.simons} first. By \cite[Lemma C.2]{LiZhouZhu}, we have that
	\[\left|\frac{\partial |A|}{\partial \eta}\right| \le 3\sqrt{n-1} |\cot\theta| |A|^2.\]
	Thus, using \eqref{eq.trace} we estimate
	\begin{align*}
	\left|\cot\theta \int_{\partial M} \right.& \left.A(\eta,\eta)|A|^{2p}f^2 - \int_{\partial M} f^2 \frac{|A|^{2p}}{\eta}\right| \\
	&\le c(n) |\cot \theta| \int_{\partial M} f^2 |A|^{2p+1}\\
	&\le c(n) \left|\frac{\cot\theta}{\sin\theta}\right| \int_M |\nabla (f^2 |A|^{2p+1})|\\
	&\le c(n) \left|\frac{\cot\theta}{\sin\theta}\right| \int_M f^2 |\nabla |A||^2 |A|^{2p-2} + f^2 |A|^{2p+2} + |\nabla f|^2 |A|^{2p}. 
	\end{align*}
	Here $c(n)$ is a constant that may change from line to line, but depends only on $n$. We have used that $p \leq 1$ and the Young's inequality $2ab<a^2 + b^2$ above.
	
	Plug this into \eqref{eq.f|A|^q.in.simons}, and use \eqref{eq.modified.simons}, we have
	\begin{align}
	0 & \leq \int_M \left( (p-1) + c(n) \left| \frac{\cot\theta}{\sin\theta} \right| \right) |A|^{2p+2} f^2 \nonumber \\
	&+ \int_M \left( p^2 - p ( (1-\lambda)(1+2/n) + \lambda + 2(p-1)) + c(n) \left| \frac{\cot\theta}{\sin\theta} \right|\right) |A|^{2p-2} |\nabla |A||^2 f^2 \nonumber \\
	& + \int_M \left(1 + c(n) \left| \frac{\cot\theta}{\sin\theta} \right| \right) |A|^{2p} |\nabla f|^2 - 2\lambda p |A|^{2p} f^2/r^2 \nonumber \\
	&=: I + II + III. \label{eq.simons.for.rigidity}
	\end{align}
	
	For $\eps>0$ to be chosen later, set 
	\[f=\begin{cases}
	r^{1+\eps} \quad r\le 1 \\ r^{2-n/2-\eps} \quad r>1.
	\end{cases} \]
	Though $f$ is not compactly supported, $\int_0^\infty r^{n-3-2p}f(r) dr$ is finite (since $p \in [1/2, 1]$ and $n \geq 2$), and thus the right hand side of \eqref{eq.f|A|^q.in.simons} is finite with this choice of $f$.  Morever, we see that $III$ becomes
	\begin{align*}
	III &= \int_{M \cap \{ r \leq 1\}} \left( (1 + c(n) \left| \frac{\cot\theta}{\sin\theta} \right|)(1+\eps)^2 - 2\lambda p \right) |A|^{2p} f^2/r^2 \\
	&\quad+ \int_{M \cap \{ r \geq 1\}} \left( (1+c(n)\left|\frac{\cot\theta}{\sin\theta}\right|)(2-n/2-\eps)^2 - 2\lambda p\right) |A|^{2p} f^2/r^2 .
	\end{align*}
	
	If we can make the coefficients in integrals I, II, II all $< 0$, it will follow that $A \equiv 0$, proving our Theorem.  We therefore seek to find $p \in (1/2, 1), \eps \in (0, 1), \lambda \in (0, 1), \theta_1 > 0$ to ensure the following inqualities hold whenever $\theta \in (\pi/2 - \theta_1, \pi/2]$ and $2 \leq n \leq 6$:
	\begin{align}
	& p-1 + c(n) \left| \frac{\cot\theta}{\sin\theta}\right| < 0 \\
	& p^2 - p( (1-\lambda)(1+2/n) + \lambda + 2p - 2) + c(n) \left| \frac{\cot\theta}{\sin\theta} \right| < 0 \\
	& (1+c(n) \left| \frac{\cot\theta}{\sin\theta} \right|) (1+\eps)^2 - 2\lambda p < 0 \\
	& (1+c(n) \left| \frac{\cot\theta}{\sin\theta} \right| ) (2-n/2-\eps)^2 - 2\lambda p < 0.
	\end{align}
	
	One can readily check that when $\theta = \pi/2$ (and $2 \leq n \leq 6$), then taking $p = 1-\eps$, $\lambda = 1 - 10\eps$, $\eps = 1/100$ will make all of the above inequalities true.  Therefore by continuity there is a $\theta_1 > 0$ so that all inequalities remain sharp if $\theta \in (\pi/2-\theta_1, \pi/2]$.  This finishes the proof. 
\end{proof}


\bibliography{bib}
\bibliographystyle{amsplain}
\end{document}